\documentclass[a4paper,11pt]{article}

\usepackage[T1]{fontenc}
\usepackage[latin9]{inputenc}
\usepackage{float}
\usepackage{amsmath}
\usepackage{graphicx}
\usepackage{amssymb}
\usepackage{framed}
\usepackage{subfigure}

\usepackage[colorlinks=true, linkcolor=black, citecolor=black,]{hyperref} 
\usepackage[english]{babel}
 \makeatletter
 \@addtoreset{equation}{section}
 \makeatother

 \makeatletter
 \@addtoreset{enunciato}{section}
 \makeatother

 \newcounter{enunciato}[subsection]

 \newtheorem{ittheorem}{Theorem}
 \newtheorem{itlemma}{Lemma}
 \newtheorem{itproposition}{Proposition}
 \newtheorem{itdefinition}{Definition}
 \newtheorem{itremark}{Remark}
 \newtheorem{itclaim}{Claim}
 \newtheorem{itfact}{Fact}
 \newtheorem{itconjecture}{Conjecture}

 \newenvironment{theorem}{\addtocounter{enunciato}{1}
 \begin{ittheorem}}{\end{ittheorem}}

 \newenvironment{lemma}{\addtocounter{enunciato}{1}
 \begin{itlemma}}{\end{itlemma}}

 \newenvironment{proposition}{\addtocounter{enunciato}{1}
 \begin{itproposition}}{\end{itproposition}}

 \newenvironment{definition}{\addtocounter{enunciato}{1}
 \begin{itdefinition}}{\end{itdefinition}}

 \newenvironment{remark}{\addtocounter{enunciato}{1}
 \begin{itremark}}{\end{itremark}}

 \newenvironment{claim}{\addtocounter{enunciato}{1}
 \begin{itclaim}}{\end{itclaim}}

 \newenvironment{fact}{\addtocounter{enunciato}{1}
 \begin{itfact}}{\end{itfact}}

 \newenvironment{conjecture}{\addtocounter{enunciato}{1}
 \begin{itconjecture}}{\end{itconjecture}}

 \newcommand{\be}[1]{\begin{equation}\label{#1}}
 \newcommand{\ee}{\end{equation}}

 \newcommand{\bl}[1]{\begin{lemma}\label{#1}}
 \newcommand{\el}{\end{lemma}}

 \newcommand{\br}[1]{\begin{remark}\label{#1}}
 \newcommand{\er}{\end{remark}}

 \newcommand{\bt}[1]{\begin{theorem}\label{#1}}
 \newcommand{\et}{\end{theorem}}

 \newcommand{\bd}[1]{\begin{definition}\label{#1}}
 \newcommand{\ed}{\end{definition}}

 \newcommand{\bcl}[1]{\begin{claim}\label{#1}}
 \newcommand{\ecl}{\end{claim}}

 \newcommand{\bfact}[1]{\begin{fact}\label{#1}}
 \newcommand{\efact}{\end{fact}}

 \newcommand{\bp}[1]{\begin{proposition}\label{#1}}
 \newcommand{\ep}{\end{proposition}}

 \newcommand{\bc}[1]{\begin{corollary}\label{#1}}
 \newcommand{\ec}{\end{corollary}}

 \newcommand{\bcj}[1]{\begin{conjecture}\label{#1}}
 \newcommand{\ecj}{\end{conjecture}}

 \newcommand{\bpr}{\begin{proof}}
 \newcommand{\epr}{\end{proof}}

 \newcommand{\bi}{\begin{itemize}}
 \newcommand{\ei}{\end{itemize}}

 \newcommand{\ben}{\begin{enumerate}}
 \newcommand{\een}{\end{enumerate}}

\def \da {\downarrow}

\newcommand{\qed}{\hspace*{\fill} $\blacksquare$\medskip}

\newenvironment{proof}
{\noindent {\em Proof}.\,\,}{\qed}

\def \b {\beta}
\def \h {\eta}
\def \t {\tau}
\def \d {\delta}

\def \Z {\mathbb Z}

\def \N {\mathbb N}

\def \Z {{\mathbb Z}}
\def \P {{\mathbb P}}
\def \E {{\mathbb E}}

\def \cX {\mathcal X}

\def \cR {\mathcal R}
\def \cC {\mathcal C}
\def \cP {\mathcal P}

\def \cI {\mathcal I}
\def \cB {\mathcal B}
\def \cF {\mathcal F}
\def \cE {\mathcal E}
\def \cA {\mathcal A}
\def \cB {\mathcal B}
\def \cS {\mathcal S}
\def \cW {\mathcal W}
\def \cG {\mathcal G}

\def \cL {\mathcal L}

\def \starred {^{\star}}
\def \dstarred {^{\star\star}}

\def \Box {\square}
\def \Boxplus {\boxplus}
\def \CAPA {{\rm CAP}}
\def \CS {{\rm CS}}

\def \supp {\mathrm{supp}}

\def \metaset	{\cX_\mathrm{meta}}
\def \groundset {\cX_\mathrm{stab}}
\newcommand 	{\stablev}[1] {V_{#1}}
\def \comlev 	{\varPhi}
\newcommand 	{\lowset}[1] {\cI_{#1}}

\def \gate {\cC^{\star}}
\def \proto {\cP}
\def \entgate {\gate_\mathrm{bd}}

\def \ta {1}
\def \tb {2}
\def \na {n_{\ta}}
\def \nb {n_{\tb}}

\def \Da {\Delta_1}
\def \Db {\Delta_2}

\begin{document}

\author{
\renewcommand{\thefootnote}{\arabic{footnote}}
F.\ den Hollander \footnotemark[1]\,\,\,\,\footnotemark[2]
\\
\renewcommand{\thefootnote}{\arabic{footnote}}
F. R.\ Nardi \footnotemark[3]\,\,\,\,\footnotemark[2]
\\
\renewcommand{\thefootnote}{\arabic{footnote}}
A.\ Troiani \footnotemark[1]
}

\title{Metastability for Kawasaki dynamics at\\
low temperature with two types of particles}

\footnotetext[1]{
Mathematical Institute, Leiden University, P.O.\ Box 9512,
2300 RA Leiden, The Netherlands
}
\footnotetext[2]{
EURANDOM, P.O.\ Box 513, 5600 MB Eindhoven, The Netherlands
}
\footnotetext[3]{
Technische Universiteit Eindhoven, P.O.\ Box 513, 5600 MB Eindhoven,
The Netherlands
}

\maketitle

\begin{abstract}

This is the first in a series of three papers in which we study a two-dimensional
lattice gas consisting of two types of particles subject to Kawasaki dynamics at
low temperature in a large finite box with an open boundary. Each pair of particles
occupying neighboring sites has a negative binding energy provided their types are
different, while each particle has a positive activation energy that depends on
its type. There is no binding energy between neighboring particles of the same type.
At the boundary of the box particles are created and annihilated in a way that
represents the presence of an infinite gas reservoir. We start the dynamics from 
the empty box and compute the transition time to the full box. This transition is 
triggered by a \emph{critical droplet} appearing somewhere in the box.

We identify the region of parameters for which the system is metastable. For this
region, in the limit as the temperature tends to zero, we show that the first
entrance distribution on the set of critical droplets is uniform, compute the
expected transition time up to a multiplicative factor that tends to one, and 
prove that the transition time divided by its expectation is exponentially distributed. 
These results are derived under \emph{three hypotheses} on the energy landscape, 
which are verified in the second and the third paper for a certain subregion of 
the metastable region. These hypotheses involve three model-dependent quantities --
the energy, the shape and the number of the critical droplets -- which are identified 
in the second and the third paper as well.

\vskip 0.5truecm
\noindent
{\it MSC2010.} 60K35, 82C26.\\
{\it Key words and phrases.} Multi-type particle systems, Kawasaki dynamics,
metastable region, metastable transition time, critical droplet, potential
theory, Dirichlet form, capacity.\\
{\it Acknowledgment.} The authors are grateful to Anton Bovier and Gabriele dalla 
Torre for fruitful discussions.

\end{abstract}

%%%%%%%%%%%%%%%%%%%%%%%%%%%%%%%%%%%%%%%%%%%%%%%%%%%%%%%%%%%%%%%%%%%%%%%%%%%%%%%

\newpage

\section{Introduction and main results}
\label{S1}

The main motivation behind this work is to understand metastability of \emph{multi-type
particle systems} subject to \emph{conservative stochastic dynamics}. In the past ten
years a good understanding was achieved of the metastable behavior of the lattice 
gas subject to Kawasaki dynamics, i.e., random hopping of particles of a single type 
with hardcore repulsion and nearest-neighbor attraction. The analysis was based on a 
combination of techniques coming from large deviation theory, potential theory, geometry 
and combinatorics. In particular, a precise description was obtained of the time to 
nucleation (from the ``gas phase'' to the ``liquid phase''), the critical droplet triggering 
the nucleation, and the typical nucleation path, i.e., the typical growing and shrinking 
of droplets. For an overview we refer the reader to two recent papers presented at the 
12th Brazilian School of Probability: Gaudilli\`ere and Scoppola~\cite{GSnotes} and 
Gaudilli\`ere~\cite{Gnotes}. For an overview on metastability and droplet growth in a 
broader context, we refer the reader to the monograph by Olivieri and Vares~\cite{OV04}, 
and the review papers by Bovier~\cite{B09}, \cite{B11}, den Hollander~\cite{dH09}, 
Olivieri and Scoppola~\cite{OS10}.  

It turns out that for systems with two types of particles, as considered in the present
paper, the \emph{geometry} of the energy landscape is much more complex than for one 
type of particle. Consequently, it is a somewhat delicate matter to capture the proper 
mechanisms behind the growing and shrinking of droplets. Our proofs in the present 
paper use potential theory and rely on ideas developed in Bovier, den Hollander and 
Nardi~\cite{BdHN06} for Kawasaki dynamics with one type of particle. Our target is 
to identify the \emph{minimal} hypotheses that lead to metastable behavior. We will
argue that these hypotheses, stated in the context of our specific model,
 also suffice for Kawasaki dynamics with more than two
types of particles and are robust against variations of the interaction.

The model studied in the present paper falls in the class of variations on Ising spins
subject to Glauber dynamics and lattice gas particles subject to Kawasaki dynamics.
These variations include Blume--Capel, anisotropic interactions, staggered magnetic 
field, next-nearest-neighbor interactions, and probabilistic cellular automata. 
In all these models the geometry of the 
energy landscape is complex and needs to be controlled in order to arrive at a 
complete description of metastability. For an overview, see the monograph by Olivieri 
and Vares~\cite{OV04}, chapter 7.

Section~\ref{S1.1} defines the model, Section~\ref{S1.2} introduces basic notation,
Section~\ref{S1.3} identifies the metastable region, while Section~\ref{S1.4} states
the main theorems. Section~\ref{S1.5} discusses the main theorems, \emph{places them 
in their proper context and provides further motivation}. Section~\ref{S1.6} proves 
three geometric lemmas that are needed in the proof of the main theorems, which is
provided in Section~\ref{S2}.

%%%%%%%%%%%%%

\subsection{Lattice gas subject to Kawasaki dynamics}
\label{S1.1}

Let $\Lambda \subset \Z^2$ be a large finite box. Let
\begin{equation}
\label{boxinout}
\begin{aligned}
\partial\Lambda^- &= \{x\in\Lambda\colon\,\exists\,y\notin\Lambda\colon\,|y-x|=1\},\\
\partial\Lambda^+ &= \{x\notin\Lambda\colon\,\exists\,y\in\Lambda\colon\,|y-x|=1\},
\end{aligned}
\end{equation}
be the internal boundary, respectively, the external boundary of $\Lambda$, and put
$\Lambda^-=\Lambda\backslash\partial\Lambda^-$ and $\Lambda^+=\Lambda\cup\partial\Lambda^+$.
With each site $x\in\Lambda$ we associate a variable $\eta(x) \in \{0,1,2\}$ indicating
the absence of a particle or the presence of a particle of type $\ta$ or type $\tb$,
respectively. A configuration $\eta=\{\eta(x)\colon\,x\in\Lambda\}$ is an element of $\cX
=\{0,1,2\}^\Lambda$. To each configuration $\eta$ we associate an energy given by
the Hamiltonian
\begin{equation}
\label{Ham1}
H(\eta) = -U \sum_{(x,y)\in(\Lambda^-)\starred} 1_{\{\eta(x)\eta(y)=2\}}\\
+ \Da \sum_{x\in\Lambda} 1_{\{\eta(x)=1\}}
+ \Db \sum_{x\in\Lambda} 1_{\{\eta(x)=2\}},
\end{equation}
where $(\Lambda^-)\starred=\{(x,y)\colon\,x,y\in\Lambda^-,|x-y|=1\}$ is the set of
non-oriented bonds inside $\Lambda^-$ (with $|\cdot|$ the Euclidean norm), $-U<0$ is 
the \emph{binding energy} between neighboring particles of \emph{different} types inside 
$\Lambda^-$, and $\Da>0$ a
nd $\Db>0$ are the \emph{activation energies} of particles of 
type $\ta$, respectively, type $\tb$ inside $\Lambda$. Without loss of generality we will 
assume that
\begin{equation}
\label{delineq}
\Da \leq \Db.
\end{equation}
The Gibbs measure associated with $H$ is
\begin{equation}
\label{Gibbs}
\mu_\beta(\eta) = \frac{1}{Z_\beta}\,e^{-\beta H(\eta)}, \qquad \eta\in\cX,
\end{equation}
where $\beta\in (0,\infty)$ is the inverse temperature, and $Z_\beta$ is the normalizing
partition sum.

\emph{Kawasaki dynamics} is the continuous-time Markov process $(\eta_t)_{t \geq 0}$ with
state space $\cX$ whose transition rates are
\begin{equation}
\label{rate}
c_\beta(\eta,\eta') = \left\{\begin{array}{ll}
e^{-\beta [H(\eta')-H(\eta)]_+}, &\eta,\eta'\in\cX,\,\eta \sim \eta',\\
0, &\text{otherwise},
\end{array}
\right.
\end{equation}
(i.e., Metroplis rate w.r.t.\ $\beta H$), where $\eta\sim\eta'$ means that $\eta'$ can be 
obtained from $\eta$ and vice versa by one of the following moves:
\begin{itemize}
\item[$\bullet$]
interchanging the states $0 \leftrightarrow 1$ or $0 \leftrightarrow 2$ at neighboring
sites in $\Lambda$\\
(``hopping of particles inside $\Lambda$''),
\item[$\bullet$]
changing the state $0 \rightarrow 1$, $0 \rightarrow 2$, $1 \rightarrow 0$ or
$2 \rightarrow 0$ at single sites in $\partial^- \Lambda$\\
(``creation and annihilation of particles inside $\partial^- \Lambda$'').
\end{itemize}
This dynamics is ergodic and reversible with respect to the Gibbs measure $\mu_\beta$, i.e.,
\begin{equation}
\mu_\b(\eta)c_\beta(\eta,\eta') = \mu_\beta(\eta')c_\beta(\eta',\eta) \qquad 
\forall\,\eta,\eta'\in\cX.
\end{equation}
Note that particles are preserved in $\Lambda^-$, but can be created and annihilated 
in $\partial^-\Lambda$. Think of the particles entering and exiting $\Lambda$ along 
non-oriented edges between $\partial^-\Lambda$ and $\partial^+\Lambda$ (where we allow 
only one edge for each site in $\partial^-\Lambda$). The pairs $(\eta,\eta')$ with 
$\eta\sim\eta'$ are called \emph{communicating configurations}, the transitions 
between them are called \emph{allowed moves}. Note that particles in $\partial^-\Lambda$ 
do not interact with particles anywhere in $\Lambda$ (see \eqref{Ham1}).

The dynamics defined by (\ref{Ham1}) and (\ref{rate}) models the behavior inside
$\Lambda$ of a lattice gas in $\Z^2$, consisting of two types of particles subject
to random hopping with hard core repulsion and with binding between different
neighboring types. We may think of $\Z^2\backslash\Lambda$ as an \emph{infinite reservoir}
that keeps the particle densities inside $\Lambda$ fixed at $\rho_\ta=e^{-\beta\Da}$ 
and $\rho_\tb=e^{-\beta\Db}$. In our model this reservoir is replaced by an \emph{open
boundary} $\partial^-\Lambda$, where particles are created and annihilated at a rate that
matches these densities. Consequently, our Kawasaki dynamics is a \emph{finite-state}
Markov process.

Note that there is \emph{no} binding energy between neighboring particles of the
\emph{same} type. Consequently, the model does \emph{not} reduce to Kawasaki 
dynamics for one type of particle when $\Da = \Db$. Further note that, whereas 
Kawasaki dynamics for one type of particle can be interpreted as swaps of occupation
numbers along edges, such an interpretation is not possible here. 

%%%%%%%%%%%%%%%%%%%%

\subsection{Notation}
\label{S1.2}

To identify the metastable region in Section~\ref{S1.3} and state our main theorems
in Section~\ref{S1.4}, we need some notation.

\begin{definition}
\label{def1}
(a) $n_i(\eta)$ is the number of particles of type $i=\ta,\tb$ in $\eta$.\\
(b) $B(\eta)$ is the number of bonds in $(\Lambda^-)\starred$ connecting neighboring particles 
of different type in $\eta$, i.e., the number of active bonds in $\eta$.\\
(c) A droplet is a maximal set of particles connected by active bonds.\\
(d) $\Box$ is the configuration where $\Lambda$ is empty, $\Boxplus$ is the configuration 
where $\Lambda$ is filled as a checkerboard (see Remark~{\rm \ref{bdeff}} below).\\
(e) $\omega\colon\,\eta\to\eta'$ is any path of allowed moves from $\eta$ to $\eta'$.\\
(f) $\tau_\cA$ = $\inf\{t\geq 0\colon\,\eta_t \in \cA,\,\exists\,0<s<t\colon \eta_s\notin \cA\}$, 
$\cA\subset\cX$, is the first hitting/return time of $\cA$.\\
(g) $\P_\eta$ is the law of $(\eta_t)_{t \geq 0}$ given $\eta_0=\eta$.
\end{definition}

\begin{definition}
\label{def2}
(a) $\comlev(\eta,\eta')$ is the communication height between $\eta,\eta'\in\cX$
defined by
\begin{equation}
\comlev(\eta,\eta') = \min_{\omega\colon\,\eta\rightarrow\eta'}
\max_{\xi\in\omega} H(\xi),
\end{equation}
and $\comlev(\cA,\cB)$ is its extension to non-empty sets $\cA,\cB\subset\cX$ defined by
\begin{equation}
\comlev(\cA,\cB) = \min_{\eta\in \cA,\eta'\in \cB} \comlev(\eta,\eta').
\end{equation}
(b) $\cS(\eta,\eta')$ is the communication level set between $\eta$ and $\eta'$
defined by
\begin{equation}
\cS(\eta,\eta') = \left\{\zeta\in\cX\colon\,\exists\,\omega\colon\,\eta\to\eta',\,
\omega\ni\zeta\colon\,\max_{\xi\in\omega} H(\xi) = H(\zeta) = \Phi(\eta,\eta')\right\}.
\end{equation}
(c) $\stablev{\eta}$ is the stability level of $\eta\in\cX$ defined by
\begin{equation}
\stablev{\eta} = \comlev({\eta},\lowset{\eta})- H(\eta),
\end{equation}
where $\lowset{\eta}=\{\xi\in\cX\colon\,H(\xi)<H(\eta)\}$ is the set of
configurations with energy lower than $\eta$.\\
(d) $\groundset=\{\eta\in\cX\colon\,H(\eta)=\min_{\xi\in\cX} H(\xi)\}$
is the set of stable configurations, i.e., the set of configurations with minimal
energy.\\
(e) $\metaset=\{\eta\in\cX\colon\,V_\eta = \max_{\xi\in\cX\backslash\groundset}
V_\xi\}$ is the set of metastable configurations, i.e., the set of non-stable
configurations with maximal stability level.\\
(f) $\Gamma=\stablev{\eta}$ for $\eta\in\metaset$ (note that $\eta\mapsto\stablev{\eta}$
is constant on $\metaset$), $\Gamma\starred=\comlev(\Box,\Boxplus)-H(\Box)$ (note that
$H(\Box)=0$).
\end{definition}

\begin{definition}
\label{def3}
(a) $(\eta\to\eta')_\mathrm{opt}$ is the set of paths realizing the minimax in
$\Phi(\eta,\eta')$.\\
(b) A set $\cW\subset\cX$ is called a gate for $\eta\to\eta'$ if $\cW\subset
\cS(\eta,\eta')$ and $\omega\cap\cW\neq\emptyset$ for all $\omega\in
(\eta\to\eta')_\mathrm{opt}$.\\
(c) A set  $\cW\subset\cX$ is called a minimal gate for $\eta\to\eta'$ if
it is a gate for $\eta\to\eta'$ and for any $\cW'\subsetneq\cW$ there exists an
$\omega' \in (\eta\to\eta')_\mathrm{opt}$ such that $\omega'\cap\cW'=\emptyset$.\\
(d) A priori there may be several (not necessarily disjoint) minimal gates.
Their union is denoted by $\cG(\eta,\eta')$ and is called the essential gate
for $(\eta\to\eta')_\mathrm{opt}$. (The configurations in $\cS(\eta,\eta')\backslash
\cG(\eta,\eta')$ are called dead-ends.)
\end{definition}

Definitions~\ref{def2}--\ref{def3} are canonical in metastability theory and 
are formalized in Manzo, Nardi, Olivieri and Scoppola~\cite{MNOS04}.

%%%%%%%%%%%%%%%%%%%%

\subsection{Metastable region}
\label{S1.3}

We want to understand how the system tunnels from $\Box$ to $\Boxplus$ when the former
is a local minimum and the latter is a global minimum of $H$. We begin by identifying the
\emph{metastable region}, i.e., the region in parameter space for which this is the case.

\begin{lemma}
\label{lmetreg}
The condition $\Da+\Db<4U$ is necessary and sufficient for $\Box$ to be
a local minimum but not a global minimum of $H$.
\end{lemma}

\begin{proof}
Note that $H(\Box)=0$. We know that $\Box$ is a local minimum of $H$, since as soon
as a particle enters $\Lambda$ we obtain a configuration with energy either $\Da>0$
or $\Db>0$. To show that there is a configuration $\hat\eta$ with $H(\hat\eta)<0$, we
write
\begin{equation}
H(\eta) = \na(\eta)\Da + \nb(\eta) \Db - B(\eta) U.
\end{equation}
Since $\Da \leq \Db$, we may assume without loss of generality that $\na(\eta)\geq\nb(\eta)$. 
Indeed, if $\na(\eta)<\nb(\eta)$, then we simply take the configuration $\eta^{\ta
\Leftrightarrow\tb}$ obtained from $\eta$ by interchanging the types $\ta$ and $\tb$, i.e.,
\begin{equation}
\label{12map}
\eta^{\ta \Leftrightarrow \tb}(x) =	
\begin{cases}
\ta & \text{if } \eta(x) = \tb, \\
\tb & \text{if } \eta(x) = \ta, \\
0 & \text{otherwise},
\end{cases}
\end{equation}
which satisfies $H(\eta^{\ta \Leftrightarrow \tb}) \leq H(\eta)$.

Since $B(\eta) \leq 4\nb(\eta)$, we have
\begin{equation}
H(\eta) \geq \na(\eta)\Da + \nb(\eta)\Db - 4\nb(\eta)U \geq \nb(\eta)(\Da+\Db-4U).
\end{equation}
Hence, if $\Da+\Db \geq 4U$, then $H(\eta) \geq 0$ for all $\eta$ and $H(\Box)=0$
is a global minimum. On the other hand, consider a configuration $\hat\eta$ such
that $\na(\hat\eta) = \nb(\hat\eta)$ and $\na(\hat\eta) + \nb(\hat\eta) = \ell^2$
for some $\ell\in 2\N$. Arrange the particles of $\hat\eta$ in a checkerboard square
of side length $\ell$. Then a straightforward computation gives
\begin{equation}
H(\hat{\eta}) =\tfrac12\ell^2\Da+\tfrac12\ell^2\Db-2\ell(\ell-1)U,
\end{equation}
and so
\begin{equation}
H(\hat{\eta})<0 \Longleftrightarrow \ell^2(\Da+\Db)<4\ell(\ell-1)U
\Longleftrightarrow \Da+\Db < (4 - 4\ell^{-1})U.
\end{equation}
Hence, if $\Da+\Db<4U$, then there exists an $\bar\ell\in2\N$ such that $H(\hat\eta)<0$
for all $\ell\in 2\N$ with $\ell\geq\bar\ell$. Here, $\Lambda$ must be taken large
enough, so that a droplet of size $\bar\ell$ fits inside $\Lambda^-$.
\end{proof}

\noindent
Note that $\Gamma\starred=\Gamma\starred(U,\Da,\Db) \in (0,\infty)$ because of 
Lemma~\ref{lmetreg}.

Within the metastable region $\Da+\Db<4U$, we may as well exclude the subregion $\Da,\Db<U$
(see Fig.~\ref{fig-propmetreg}). In this subregion, each time a particle of type $\ta$ 
enters $\Lambda$ and attaches itself to a particle of type $\tb$ in the droplet, or 
vice versa, the energy goes down. Consequently, the ``critical droplet'' for the 
transition from $\Box$ to $\Boxplus$ consists of only two free particles, one of type 
$\ta$ and one of type $\tb$. Therefore this subregion does not exhibit proper metastable 
behavior.

%%%%%%%%%%%%%%%%%%%%%%%%%%%%%%%%%%%%%%%%%%%%%%%%%%%%%%%%%%%%%%%%%%%%%%%%%%%%%%%%%
\begin{figure}[H]
\begin{centering}
{\includegraphics[width=0.35\textwidth]{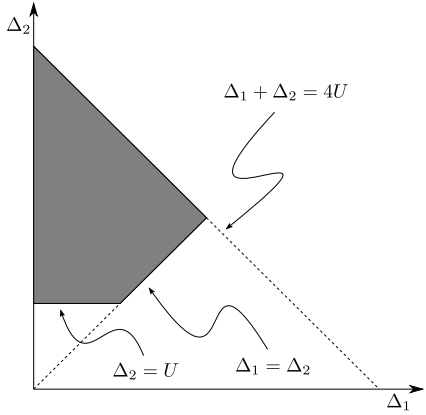}}
\par\end{centering}
\caption{Proper metastable region.}
\label{fig-propmetreg}
\end{figure}
%%%%%%%%%%%%%%%%%%%%%%%%%%%%%%%%%%%%%%%%%%%%%%%%%%%%%%%%%%%%%%%%%%%%%%%%%%%%%%%%%%%%%%%%%

%%%%%%%%%%%%%%%%%%%%%%%%%%%%%%%%%%%%%%%%%%%%%%%%%%%%%%%%%%%%%%%%%%%%%%%%%%%%%%%%%%%%%%%%%%

\subsection{Main theorems}
\label{S1.4}

Theorems~\ref{tgate}--\ref{texp} below will be proved in the \emph{metastable region} 
subject to the following \emph{hypotheses}:
\begin{itemize}
\item[(H1)]
$\groundset=\Boxplus$.
\item[(H2)]
There exists a $V^\star<\Gamma\starred$ such that $V_{\eta}\leq V^\star$ for all $\eta\in
\cX\backslash\{\Box,\Boxplus\}$.
\end{itemize}
The third hypothesis consists of three parts characterizing the entrance set of $\cG(\Box,
\boxplus)$, the set of critical droplets, and the exit set of $\cG(\Box,\boxplus)$. To 
formulate this hypothesis some further definitions are needed.

\begin{definition}
\label{defdroplets-a}
(a) $\entgate$ is the minimal set of configurations in $\cG(\Box,\Boxplus)$ such that all 
paths in $(\Box\to\Boxplus)_\mathrm{opt}$ enter $\cG(\Box,\Boxplus)$ through $\entgate$.\\
(b) $\proto$ is the set of configurations visited by these paths just prior to their first
entrance of $\cG(\Box,\Boxplus)$.
\end{definition}

\begin{itemize}
\item[(H3-a)]
Every $\hat{\eta}\in\proto$ consists of a \emph{single droplet} somewhere in $\Lambda^-$. 
This single droplet fits inside an $L\starred \times L\starred$ square somewhere in 
$\Lambda^-$ for some $L\starred \in \N$ large enough that is independent of $\hat{\eta}$ 
and $\Lambda$. Every $\eta\in\entgate$ consists of a single droplet $\hat{\eta}\in\proto$ 
and a \emph{free particle} of type $\tb$ somewhere in $\partial^-\Lambda$.
\end{itemize}

\begin{definition}
\label{defdroplets-b}
(a) $\gate_\mathrm{att}$ is the set of configurations obtained from $\proto$ by attaching 
a particle of type $\tb$ to the single droplet, and decomposes as $\gate_\mathrm{att} 
= \cup_{\hat{\eta}\in\proto} \gate_\mathrm{att}(\hat{\eta})$.\\
(b) $\gate$ is the set of configurations obtained from $\proto$ by adding a free particle 
of type $\tb$ somewhere in $\Lambda$, and decomposes as $\gate = \cup_{\hat{\eta}\in\proto} 
\gate(\hat{\eta})$. 
\end{definition}

Note that $\Gamma\starred=H(\gate)=H(\proto)+\Db$, and that $\gate$ consists of precisely 
those configurations ``interpolating'' between $\proto$ and $\gate_\mathrm{att}$: a free 
particle of type $\tb$ enters $\partial^-\Lambda$ and moves to the single droplet where 
it attaches itself via an active bond. Think of $\proto$ as the set of configurations 
where the dynamics is ``almost over the hill'', of $\gate$ as the set of configurations 
where the dynamics is ``on top of the hill'', and of the free particle as ``achieving 
the crossover'' before it attaches itself properly to the single droplet (the meaning of 
the word properly will become clear in Section~\ref{S2.3}). 

The set $\proto$ is referred to as the set of \emph{protocritical droplets}. We write 
$N\starred$ to denote the cardinality of $\proto$ modulo shifts of the droplet. The set
$\gate$ is referred to as the set of \emph{critical droplets}.

\begin{itemize}
\item[(H3-b)] 
All transitions from $\gate$ that either add a particle in $\Lambda$ or increase the number 
of droplets (by breaking an active bond) lead to energy $>\Gamma\starred$.
\item[(H3-c)] 
All $\omega \in (\entgate \to \boxplus)_{\mathrm{opt}}$ pass through $\gate_{\mathrm{att}}$.
For every $\hat\eta \in \proto$ there exists a $\zeta \in \gate_{\mathrm{att}}(\hat\eta)$ such 
that $\Phi(\zeta,\boxplus)<\Gamma\starred$.
\end{itemize}

We are now ready to state our main theorems subject to (H1)--(H3).

\begin{theorem}
\label{tgate}
(a) $\lim_{\beta\to\infty} P_\Box(\tau_{\entgate}<\tau_\Boxplus \mid \tau_\Boxplus
<\tau_\Box)=1$.\\
(b) $\lim_{\beta\to\infty} P_\Box(\eta_{\tau_{\entgate}}=\zeta)
= 1/|\entgate|$ for all $\zeta \in \entgate$.
\end{theorem}

\begin{theorem}
\label{tnucltime}
There exists a constant $K=K(\Lambda;U,\Da,\Db) \in (0,\infty)$ such that
\begin{equation}
\label{sharpasymp}
\lim_{\beta\to\infty} e^{-\beta\Gamma\starred} E_{\Box}(\tau_\Boxplus) = K.
\end{equation}
Moreover,
\begin{equation}
\label{Kasymp}
K \sim \frac{1}{N\starred}\,\frac{\log|\Lambda|}{4\pi|\Lambda|}
\qquad \mbox{ as } \Lambda \to \Z^2.
\end{equation}
\end{theorem}

\begin{theorem}
\label{texp}
$\lim_{\beta\to\infty} P_\Box(\tau_\Boxplus/E_\Box(\tau_\Boxplus)>t)
= e^{-t}$ for all $t \geq 0$.
\end{theorem}

We close this section with a few remarks.

\br{Tatbfree}
{\rm The free particle in (H3-a) is of type $\tb$ only when $\Da<\Db$. If $\Da=\Db$ 
(recall \eqref{delineq}), then the free particle can be of type $\ta$ or $\tb$. Indeed, 
for $\Da=\Db$ there is full symmetry of $\cS(\Box,\Boxplus)$ under the map $\ta 
\Leftrightarrow \tb$ defined in \eqref{12map}.}
\er

\br{Hcons}
{\rm We will see in Section~\ref{S1.6} that (H1--H2) imply that}
\begin{equation}
(\metaset,\groundset) = (\Box,\Boxplus), \qquad \Gamma=\Gamma\starred.
\end{equation}
{\rm The reason that $\Boxplus$ is the configuration with lowest energy comes
from the ``anti-ferromagnetic'' nature of the interaction in \eqref{Ham1}.}
\er

\br{Hyprel}
{\rm Note that (H2) and Lemma~\ref{lmetreg} imply (H1). Indeed, (H2) says that $\Box$ 
and $\Boxplus$ have the highest stability level in the sense of Definition~\ref{def2}(c),
so that $\groundset\subset\{\Box,\Boxplus\}$, while Lemma~\ref{lmetreg} says that $\Box$ 
is not the global minimum of $H$, so that $\Boxplus$ must be the global minumum of $H$, 
and hence $\groundset=\Boxplus$ according to Definition~\ref{def2}(d). One reason why we
state (H1)--(H2) as separate hypotheses is that we will later place them in a more 
general context (see Section~\ref{S1.5}, item 8). Another reason is that they are the 
key ingredients in the proof of Theorems~\ref{tgate}--\ref{texp} in Section~\ref{S2}.} 
\er

\br{bdeff}
{\rm We will see in \cite{dHNTpr1} that, depending on the shape of $\Lambda$ and the 
choice of $U,\Da,\Db$, $\groundset$ may actually consist of more than the single 
configuration $\Boxplus$, namely, it may contain configurations that differ from 
$\Boxplus$ in $\partial^-\Lambda$. Since this boundary effect does not affect our main 
theorems, we will ignore it here. A precise description of $\groundset$ will be given 
in \cite{dHNTpr1}. Moreover, depending on the choice of $U,\Da,\Db$, large droplets 
with minimal energy tend to have a shape that is either \emph{square-shaped} or 
\emph{rhombus-shaped}. Therefore it turns out to be expedient to choose $\Lambda$ 
to have the same shape. Details will be given in \cite{dHNTpr1}.}
\er

\br{H3asymp}
{\rm As we will see in Section~\ref{S2.3}, the value of $K$ is given by a \emph{non-trivial 
variational formula} involving the set of all configurations where the dynamics can enter 
and exit $\gate$. This set includes not only the border of the ``$\Gamma\starred$-valleys'' 
around $\Box$ and $\Boxplus$, but also the border of ``wells inside the energy plateau 
$\cG(\Box,\Boxplus)$'' that have energy $<\Gamma\starred$ but communication height $\Gamma\starred$ towards 
both $\Box$ and $\Boxplus$. This set contains $\proto$, $\gate_\mathrm{att}$ and possibly 
more, as we will see in \cite{dHNTpr2} (for Kawasaki dynamics with one type of particle 
this was shown in Bovier, den Hollander and Nardi~\cite{BdHN06}, Section~2.3.2). As a 
result of this geometric complexity, for finite $\Lambda$ only upper and lower bounds 
are known for $K$. What (\ref{Kasymp}) says is that these bounds merge and simplify in 
the limit as $\Lambda\to\Z^2$ (after the limit $\beta\to\infty$ has already been taken), 
and that for the asymptotics only the simpler quantity $N\starred$ matters rather than 
the full geometry of critical and near critical droplets. We will see in Section~\ref{S2.3} 
that, apart from the uniformity property expressed in Theorem~\ref{tgate}(b), the reason 
behind this simplification is the fact that simple random walk (the motion of the free 
particle) is \emph{recurrent} on $\Z^2$.}
\er

%%%%%%%%%%%%%%%%%%%%%%%%%%%%%%%%%%%%%%%%%%%%%%%%%%%%%%%%%%%%%%%%%%%%%%%%%%%%%%%%%%

\subsection{Discussion}
\label{S1.5}

{\bf 1.} 
Theorem~\ref{tgate}(a) says that $\gate$ is a gate for the nucleation, i.e., on its way 
from $\Box$ to $\boxplus$ the dynamics passes through $\gate$. Theorem~\ref{tgate}(b) says 
that all protocritical droplets and all locations of the free particle in $\partial^-\Lambda$
are equally likely to be seen upon first entrance in $\cG(\Box,\Boxplus)$. Theorem~\ref{tnucltime} 
says that the average nucleation time is asymptotic to $Ke^{\Gamma\beta}$, which is the 
classical Arrhenius law, and it identifies the asymptotics of the prefactor $K$ in the limit 
as $\Lambda$ becomes large. Theorem~\ref{texp}, finally, says that the nucleation time is 
exponentially distributed on the scale of its average.

\medskip\noindent
{\bf 2.}
Theorems~\ref{tgate}--\ref{texp} are \emph{model-independent}, i.e., they are expected 
to hold in the same form for a large class of stochastic dynamics in a finite box at 
low temperature exhibiting metastable behavior. So far this universality has been verified 
for only a handful of examples, including Kawasaki dynamics with one type of particle 
(see also item 4 below). In Section~\ref{S2} we will see that (H1)--(H3) are the 
\emph{minimal hypotheses} needed for metastable behavior, in the sense that any relative 
of Kawasaki dynamics for which Theorems~\ref{tgate}--\ref{texp} hold must satisfy appropriate analogues of
(H1)--(H3) (including multi-type Kawasaki dynamics).      

The \emph{model-dependent} ingredient of Theorems~\ref{tgate}--\ref{texp} is the triple
\begin{equation}
(\Gamma\starred,\gate,N\starred).
\end{equation}
This triple depends on the parameters $U,\Da,\Db$ in a manner that will be identified
in \cite{dHNTpr1} and \cite{dHNTpr2}. The set $\gate$ also depends on $\Lambda$, but 
in such a way that $|\gate|\sim N\starred|\Lambda|$ as $\Lambda\to\Z^2$, with the error 
coming from boundary effects. Clearly, $\Lambda$ must be taken large enough so that 
critical droplets fit inside (i.e., $\Lambda$ must contain an $L\starred \times L\starred$ 
square with $L\starred$ as in (H3-a)).

%%%%%%%%%%%%%%%%%%%%%%%%%%%%%%%%%%%%%%%%%%%%%%%%%%%%%%%%%%%%%%%%%%%%%%%%%%%%%%%%%%%%%
\begin{figure}[H]
\begin{centering}
{\includegraphics[width=0.35\textwidth]{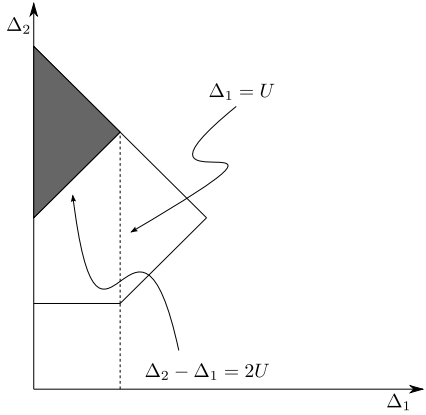}}
\par\end{centering}
\caption{Subregion of the proper metastable region considered in \cite{dHNTpr1} 
and \cite{dHNTpr2}.}
\label{fig-subpropmetreg}
\end{figure}
%%%%%%%%%%%%%%%%%%%%%%%%%%%%%%%%%%%%%%%%%%%%%%%%%%%%%%%%%%%%%%%%%%%%%%%%%%%%%%%%%%%%%%

\medskip\noindent
{\bf 3.}
In \cite{dHNTpr1} and \cite{dHNTpr2}, we will prove (H1)--(H3), identify $(\Gamma\starred,
\gate,N\starred)$ and derive an upper bound on $V^\star$ in the subregion of the proper
metastable region given by (see Fig.~\ref{fig-subpropmetreg})
\begin{equation}
\label{Dres}
0<\Da<U, \quad \Db-\Da>2U.
\end{equation}
More precisely, in \cite{dHNTpr1} we will prove (H1), identify $\Gamma\starred$, show that
$V^\star\leq 10U-\Da$, and conclude that (H2) holds as soon as $\Gamma\starred>10U-\Da$, 
which poses further restrictions on $U,\Da,\Db$ on top of \eqref{Dres}. In \cite{dHNTpr1} 
we will also see that it would be possible to show that $V^\star\leq 4U+\Da$ provided 
certain boundary effects (arising when a droplet sits close to $\partial^-\Lambda$ or 
when two or more droplets are close to each other) could be controlled. Since it will 
turn out that $\Gamma\starred>4U+\Da$ throughout the region \eqref{Dres}, this upper 
bound would settle (H2) without further restrictions on $U,\Da,\Db$. In \cite{dHNTpr2} 
we will prove (H3) and identify $\gate,N\starred$.

The simplifying features of \eqref{Dres} are the following: $\Da<U$ implies that each
time a particle of type $\ta$ enters $\Lambda$ and attaches itself to a particle of
type 2 in the droplet the energy goes down, while $\Db-\Da>2U$ implies that no
particle of type 2 sits on the boundary of a droplet that has minimal energy given 
the number of particles of type 2 in the droplet. We \emph{conjecture} that (H1)--(H3) 
hold throughout the proper metastable region (see Fig.~\ref{fig-propmetreg}). However, 
as we will see in \cite{dHNTpr1} and \cite{dHNTpr2}, $(\Gamma\starred,\gate,N\starred)$ 
is different when $\Da>U$ compared to when $\Da<U$ (because the critical droplets are 
square-shaped, respectively, rhombus-shaped).

\medskip\noindent
{\bf 4.}
Theorems~\ref{tgate}--\ref{texp} generalize what was obtained for Kawasaki dynamics 
with one type of particle in den Hollander, Olivieri and Scoppola~\cite{dHOS00}, and 
Bovier, den Hollander and Nardi~\cite{BdHN06}. In these papers, the analogues of 
(H1)--(H3) were proved, ($\Gamma\starred,\gate,N\starred)$ was identified, and bounds 
on $K$ were derived that become sharp in the limit as $\Lambda\to\Z^2$. What makes the
model with one type of particle more tractable is that the stochastic dynamics
follows a \emph{skeleton} of subcritical droplets that are \emph{squares or
quasi-squares}, as a result of a \emph{standard isoperimetric inequality} for
two-dimensional droplets. For the model with two types of particles this tool
is no longer applicable and the geometry is much harder, as will become clear in
\cite{dHNTpr1} and \cite{dHNTpr2}.

Similar results hold for Ising spins subject to \emph{Glauber dynamics}, as shown
in Neves and Schonmann~\cite{NS91}, and Bovier and Manzo~\cite{BM02}. For
this system, $K$ has a simple explicit form. Theorems~\ref{tgate}--\ref{texp}
are close in spirit to the extension for Glauber dynamics of Ising spins when an alternating
external field is included, as carried out in Nardi and Olivieri~\cite{NO96},
for Kawakasi dynamics of lattice gases with one type of particle when the
interaction between particles is different in the horizontal and the vertical
direction, as carried out in Nardi, Olivieri and Scoppola~\cite{NOS05}, and for Glauber dynamics with 
three--state spins (Blume--Capel model), as carried out in Cirillo and Olivieri~\cite{CO96} 

Our results can in principle be extended from $\Z^2$ to $\Z^3$. For one type
of particle this extension was achieved in den Hollander, Nardi, Olivieri and
Scoppola~\cite{dHNOS03}, and Bovier, den Hollander and Nardi~\cite{BdHN06}. 
For one type of particle the geometry of the critical droplet is more complex 
in $\Z^3$ than in $\Z^2$. This will also be the case for two types of particles, 
and hence it will be hard to identify $\gate$ and $N\starred$. Again, only 
upper and lower bounds can be derived for $K$. Moreover, since simple random walk 
on $\Z^3$ is \emph{transient}, these bounds do \emph{not} merge in the limit as 
$\Lambda\to\Z^3$. For Glauber dynamics the extension from $\Z^2$ to $\Z^3$ was 
achieved in Ben Arous and Cerf~\cite{BAC96}, and Bovier and Manzo~\cite{BM02}, 
and $K$ again has a simple explicit form.

\medskip\noindent
{\bf 5.}
In Gaudilli\`ere, den Hollander, Nardi, Olivieri and Scoppola~\cite{GdHNOS1},
\cite{GdHNOS2}, \cite{GdHNOS3}, and Bovier, den Hollander and Spitoni~\cite{BdHS10},
the result for Kawasaki dynamics (with one type of particle) on a finite box with
an open boundary obtained in den Hollander, Olivieri and Scoppola~\cite{dHOS00} and 
Bovier, den Hollander and Nardi~\cite{BdHN06} have been extended to Kawasaki dynamics 
(with one type of particle) on a \emph{large box} $\Lambda=\Lambda_\beta$ with a 
\emph{closed boundary}. The volume of $\Lambda_\beta$ grows exponentially fast with 
$\beta$, so that $\Lambda_\beta$ itself acts as a gas reservoir for the growing and 
shrinking of subcritical droplets. The focus is on the time of the first appearance 
of a critical droplet \emph{anywhere} in $\Lambda_\beta$. It turns out that the 
nucleation time in $\Lambda_\beta$ roughly equals the nucleation time in a finite 
box $\Lambda$ divided by the volume of $\Lambda_\beta$, i.e., \emph{spatial entropy} 
enters into the game. A challenge is to derive a similar result for Kawasaki dynamics 
with two types of particles.

\medskip\noindent
{\bf 6.}
The model in the present paper can be extended by introducing three binding energies
$U_{11},U_{22},U_{12}<0$ for the three different pairs of types that can occur in
a pair of neighboring particles. Clearly, this will further complicate the analysis,
and consequently both $(\metaset,\groundset)$ and $(\Gamma\starred,\gate,N\starred)$ 
will in general be different. The model is interesting even when $\Da,\Db<0$ and $U<0$, 
since this corresponds to a situation where the infinite gas reservoir is very dense and 
tends to push particles into the box. When $\Da<\Db$, particles of type $\ta$ tend 
to fill $\Lambda$ before particles of type 2 appear, but this is not the configuration 
of lowest energy. Indeed, if $\Db-\Da<4U$, then the binding energy is strong enough to 
still favor configurations with a checkerboard structure (modulo boundary effects). 
Identifying $(\Gamma\starred,\gate,N\starred)$ seems a complicated task.

\medskip\noindent
{\bf 7.} 
We will see in Section~\ref{S2} that (H1)--(H2) alone are enough to prove 
Theorems~\ref{tgate}--\ref{texp}, with the exception of the uniform entrance distribution of $\entgate$ 
and the scaling of $K$ in \eqref{Kasymp}. 
The latter require (H3) and come out of a closer analysis of 
the energy landscape near $\gate$, respectively, a \emph{variational 
formula} for $1/K$ that is derived on the basis of (H1)--(H2) alone. 

In Manzo, Nardi, Olivieri and Scoppola~\cite{MNOS04} an ``axiomatic approach'' to 
metastability similar to the one in the present paper was put forward, but the results 
that were obtained (for a general dynamics) based on hypotheses similar to (H1)--(H2) 
were cruder, e.g.\ the nucleation time was shown to be $\exp[\beta\Gamma\starred
+o(\beta)]$, which fails to capture the fine asymptotics in (\ref{sharpasymp})
and consequently also the scaling in (\ref{Kasymp}). 
Also the uniform entrance distribution was not established. These finer details come out of 
the potential-theoretic approach to metastability developed in Bovier, Eckhoff, Gayrard
and Klein~\cite{BEGK02} explained in Section~\ref{S2}.   
    
\medskip\noindent
{\bf 8.}
Hypotheses (H1)--(H3) are the \emph{minimal hypotheses} in the following sense. If we 
consider Kawasaki dynamics with more than two types of particles and/or change the details 
of the interaction (e.g.\ by adding to \eqref{Ham1} also interactions between particles 
of different type), then all that changes is that $\Box$ and $\Boxplus$ are replaced by
different configurations, while (H1)--(H2) remain the same for their new counterparts 
and (H3) remains the same for the analogues of $\proto$, $\gate$, $\entgate$ and 
$\gate_\mathrm{att}$. The proof in Section~\ref{S2} will show that 
Theorems~\ref{tgate}--\ref{texp} continue to hold under (H1)--(H3) in the new setting. 
For further reading we refer the reader to the monograph in progress by Bovier and den 
Hollander~\cite{BdHmon}.

%%%%%%%%%%%%%%%%%%%%

\subsection{Consequences of (H1)--(H3)}
\label{S1.6}

Lemmas~\ref{lemptyismeta}--\ref{lbacktrack} below are immediate consequences of
(H1)--(H3) and will be needed in the proof of Theorems~\ref{tgate}--\ref{texp} in 
Section~\ref{S2}.

\begin{lemma}
\label{lemptyismeta}
{\rm (H1)--(H2)} imply that $\stablev{\Box} = \Gamma\starred$.
\end{lemma} 

\begin{proof}
By Definitions~\ref{def2}(c--f) and (H1), $\boxplus\in\lowset{\Box}$, which implies that 
$\stablev{\Box}\leq\Gamma\starred$. We show that (H2) implies $\stablev{\Box}=\Gamma\starred$. 
The proof is by contradiction. Suppose that $\stablev{\Box}<\Gamma\starred$. Then, by 
Definition~\ref{def2}(c) and (H2), there exists an $\eta\in\lowset{\Box}\backslash\boxplus$ 
such that $\comlev(\Box,\eta)-H(\Box)<\Gamma\starred$. But, by (H2) and the finiteness of 
$\cX$, there exist an $m \in \N$ and a sequence $\eta_{0},\ldots,\eta_{m}\in\cX$ with 
$\eta_0=\eta$ and $\eta_{m}=\boxplus$ such that $\eta_{i+1}\in\lowset{\eta_{i}}$ and 
$\comlev(\eta_{i},\eta_{i+1}) \leq H(\eta_{i})+V\starred$ for $i=0,\ldots,m-1$. Therefore
\begin{equation}
\label{eststring}
\comlev(\eta,\boxplus)
\leq \max_{i=0,\ldots,m-1} \comlev(\eta_{i}, \eta_{i+1})
\leq \max_{i=0,\ldots,m-1} [H(\eta_{i}) + V\starred]
= H(\eta) + V\starred 
< H(\Box) + \Gamma\starred,
\end{equation}
where in the first inequality we use that 
\begin{equation}
\label{ultramet}
\comlev(\eta,\sigma) \leq \max\{\comlev(\eta,\xi),\comlev(\xi,\sigma)\} 
\qquad \forall\,\eta,\sigma,\xi\in\cX,
\end{equation} 
and in the last inequality that $\eta\in\lowset{\Box}$ and $V\starred<\Gamma\starred$. It 
follows that
\begin{equation}
\Gamma\starred = \comlev(\Box,\Boxplus)-H(\Box) 
\leq \max\{\comlev(\Box,\eta),\comlev(\eta,\Boxplus)\}-H(\Box) < \Gamma\starred,
\end{equation}
which is a contradiction.
\end{proof}

\begin{lemma}
\label{ltometapair}
{\rm (H2)} implies that $\comlev(\eta,\{\Box,\Boxplus\}) - H(\eta) \leq V\starred$
for all $\eta\in\cX\backslash\{\Box,\Boxplus\}$.
\end{lemma}

\begin{proof}
Fix $\eta\in\cX\backslash\{\Box,\Boxplus\}$. By (H2) and the finiteness of $\cX$, there exist 
an $m\in\N$ and a sequence $\eta_{0},\ldots,\eta_{m}\in\cX$ with $\eta_0=\eta$ and $\eta_{m}
\in\{\Box,\Boxplus\}$ such that $\eta_{i+1}\in\lowset{\eta_{i}}$ and $\comlev(\eta_{i},
\eta_{i+1})\leq H(\eta_{i})+V\starred$ for $i=0,\ldots,m-1$. Therefore, as in (\ref{eststring}),
we get $\comlev(\eta,\{\Box,\Boxplus\})\leq H(\eta) + V\starred$.
\end{proof}

\begin{lemma}
\label{lemptyisbottom}
{\rm (H1)--(H2)} imply that $H(\eta)>H(\Box)$ for all $\eta\in\cX\backslash\Box$ such that 
$\Phi(\eta,\Box) \leq \Phi(\eta,\Boxplus)$. 
\end{lemma}

\begin{proof}
By (H1), $\Boxplus\in\lowset{\eta}$ for all $\eta\neq\Boxplus$. The proof is by contradiction. 
Fix $\eta\in\cX\backslash\Box$ and suppose that $H(\eta) \leq H(\Box) = 0$. Then 
$\Box\notin\lowset{\eta}$. By (H2) and the finiteness of $\cX$, there exist an $m\in\N$ 
and a sequence $\eta_{0},\ldots,\eta_{m}\in\cX$ with $\eta_0=\eta$ and $\eta_{m}=\Boxplus$ 
such that $\eta_{i+1}\in\lowset{\eta_{i}}$ and $\comlev(\eta_{i},\eta_{i+1})\leq H(\eta_{i})
+V\starred$ for $i=0,\ldots,m-1$. Therefore, as in (\ref{eststring}), we get $\Phi(\eta,\Boxplus)
\leq H(\eta)+V\starred\leq H(\Box)+V\starred<H(\Box)+\Gamma\starred$. Hence
\begin{equation}
\begin{aligned}
\Gamma\starred 	&= \Phi(\Box,\Boxplus)-H(\Box) 
\leq \max\{\Phi(\Box,\eta),\Phi(\eta,\Boxplus)\}-H(\Box)\\
&= \max\{\Phi(\eta,\Box),\Phi(\eta,\Boxplus)\}- H(\Box)
= \Phi(\eta,\Boxplus)-H(\Box) < \Gamma\starred,
\end{aligned}
\end{equation}
which is a contradiction.	
\end{proof}

\begin{lemma}
\label{lbacktrack}
{\rm (H3a), (H3-c)} and Definition~{\rm\ref{defdroplets-b}(a)} imply that for every 
$\eta\in\gate_\mathrm{att}$ all paths in $(\eta\to\Box)_\mathrm{opt}$ pass through 
$\entgate$.
\end{lemma}

\begin{proof}
Let $\eta$ be any configuration in $\gate_\mathrm{att}$. Then, by (H3-a) and 
Definition~\ref{defdroplets-b}(a), there is a configuration $\xi$, consisting 
of a single \emph{protocritical} droplet, say, $D$ and a free particle (of type 
$\tb$) next to the border of $D$, such that $\eta$ is obtained from $\xi$ in a 
single move: the free particle attaches itself \emph{somewhere} to $D$. Now, consider any path 
starting at $\eta$, ending at $\Box$, and not exceeding energy level $\Gamma\starred$. The reverse 
of this path, starting at $\Box$ and ending at $\eta$, can be extended by the single move from 
$\eta$ to $\xi$ to obtain a path from $\Box$ to $\xi$ that is also not exceeding energy level 
$\Gamma\starred$. Moreover, this path can be further extended from $\xi$ to $\Boxplus$ without 
exceeding energy level $\Gamma\starred$ as well. To see the latter, note that, by (H3-c), there 
is \emph{some} location $x$ on the border of $D$ such that the configuration $\zeta \in 
\gate_\mathrm{att}$ consisting of $D$ with the free particle attached at $x$ is such that there 
is a path from $\zeta$ to $\Boxplus$ that stays \emph{below} energy level $\Gamma\starred$. 
Furthermore, we can move from $\xi$ (with $H(\xi)=\Gamma\starred$) to $\zeta$ (with $H(\zeta)
<\Gamma\starred$) at constant energy level $\Gamma\starred$, dropping below $\Gamma\starred$ only 
at $\zeta$, simply by moving the free particle to $x$ without letting it hit $\partial^-\Lambda$. 
(By (H3-a), there is room for the free particle to do so because $D$ fits inside an $L\starred 
\times L\starred$ square somewhere in $\Lambda^-$. Even when $D$ touches $\partial^-\Lambda$ the 
free particle can still avoid $\partial^-\Lambda$, because $x$ can never be in $\partial^-\Lambda$: 
particles in $\partial^-\Lambda$ do not interact with particles in $\Lambda^-$.) The resulting 
path from $\Box$ to $\Boxplus$ (via $\eta$, $\xi$ and $\zeta$) is a path in $(\Box\to
\Boxplus)_\mathrm{opt}$. However, by Definition~\ref{defdroplets-a}(a), \emph{any} path 
in $(\Box\to\Boxplus)_\mathrm{opt}$ must hit $\gate_\mathrm{bd}$. Hence, the piece of the 
path from $\eta$ to $\Box$ must hit $\gate_\mathrm{bd}$, because the piece of the path from 
$\eta$ to $\Boxplus$ (via $\xi$ and $\zeta$) does not.
\end{proof}

Note that Lemma~\ref{lemptyismeta} implies that $\metaset=\Box$ and $\Gamma=\Gamma\starred$
(recall Definition~\ref{def2}(e--f). 

%%%%%%%%%%%%%%%%%%%%%%%%% SECTION 2 %%%%%%%%%%%%%%%%%%%%%%%%%%%%%%%%%%%%%%%%%%%%%%

\section{Proof of main theorems}
\label{S2}

In this section we prove Theorems~\ref{tgate}--\ref{texp} subject to hypotheses
(H1)--(H3). Sections~\ref{S2.1}--\ref{Smcnt} introduce the basic ingredients, while
Sections~\ref{S2.3}--\ref{S2.5} provide the proofs.

We will follow the \emph{potential-theoretic} argument that was used in Bovier, den 
Hollander and Nardi~\cite{BdHN06} for Kawasaki dynamics with one type of particle. 
In fact, we will see that (H1)--(H3) are the \emph{minimal} assumptions needed to 
prove Theorems~\ref{tgate}--\ref{texp}.

%%%%%%%%%%%%%%%%%%%%%%%%%%%%%%%%%%%%

\subsection{Dirichlet form and capacity}
\label{S2.1}

The key ingredient of the potential-theoretic approach to metastability is the
\emph{Dirichlet form}
\begin{equation}
\label{Diri}
\cE_\beta(h) = \tfrac12 \sum_{\eta,\eta'\in\cX} \mu_\beta(\eta)c_\beta(\eta,\eta')
[h(\eta)-h(\eta')]^2, \qquad h\colon\,\cX \to [0,1],
\end{equation}
where $\mu_\beta$ is the Gibbs measure defined in (\ref{Gibbs}) and $c_\beta$ is
the kernel of transition rates defined in (\ref{rate}). Given a pair of non-empty
disjoint sets $\cA,\cB \subset \cX$, the \emph{capacity} of the pair $\cA,\cB$
is defined by
\begin{equation}
\label{capa}
\CAPA_\beta(\cA,\cB) = \min_{ {h\colon\,\cX \to [0,1]} \atop
{h|_\cA\equiv 1,h|_\cB\equiv 0} } \cE_\beta(h),
\end{equation}
where $h|_\cA\equiv 1$ means that $h(\eta)=1$ for all $\eta\in\cA$ and $h|_\cB\equiv 0$
means that $h(\eta)=0$ for all $\eta\in\cB$. The unique minimizer $h\starred_{\cA,\cB}$ of
\eqref{capa}, called the \emph{equilibrium potential} of the pair $\cA,\cB$, is given 
by
\begin{equation}
\label{h*def}
h\starred_{\cA,\cB}(\eta) = P_\eta(\tau_\cA<\tau_\cB), \qquad \eta \in \cX\backslash(\cA\cup\cB),
\end{equation}
and is the solution of the equation
\begin{equation}
\label{soleq}
\begin{aligned}
(c_{\beta} h)(\eta) &= 0, \qquad  \eta\in \cX\backslash (\cA\cup\cB),\\ 
h(\eta)&=1, \qquad  \eta\in \cA,\\
h(\eta)&=0,\qquad  \eta\in \cB,\\ 
\end{aligned}
\end{equation}
with $(c_{\beta} h)(\eta) = \sum_{\eta'\in\cX} c_\beta(\eta,\eta')h(\eta')$. Moreover,
\begin{equation}
\label{caprep}
\CAPA_\beta(\cA,\cB) = \sum_{\eta\in\cA} \mu_\beta(\eta)\,c_{\beta}(\eta,\cX\backslash\eta)
\,\P_\eta(\tau_\cB<\tau_\cA)
\end{equation}
with $c_{\beta}(\eta,\cX\backslash\eta)=\sum_{\eta'\in\cX\backslash\eta} c_{\beta}(\eta,\eta')$ 
the rate of moving out of $\eta$. This rate enters because $\tau_\cA$ is the first hitting 
time of $\cA$ after the initial configuration is left (recall Definition~\ref{def1}(f)). 
Note that the reversibility of the dynamics and (\ref{Diri}--\ref{capa}) imply 
\begin{equation}
\label{capsym}
\CAPA_\beta(\cA,\cB)=\CAPA_\beta(\cB,\cA).
\end{equation}

The following lemma establishes bounds on the capacity of two disjoint sets. These bounds 
are referred to as \emph{a priori estimates} and will serve as the starting point for more
refined estimates later on.

\begin{lemma}
\label{lemmaprbds}
For every pair of non-empty disjoint sets $\cA,\cB\subset\cX$ there exist constants 
$0<C_1 \leq C_2<\infty$ (depending on $\Lambda$ and $\cA,\cB$) such that
\begin{equation}
\label{aprbds}
C_1 \leq e^{\beta\Phi(\cA,\cB)} Z_\beta\CAPA_\beta(\cA,\cB) \leq C_2
\qquad \forall\,\beta \in (0,\infty).
\end{equation}
\end{lemma}

\begin{proof}
The proof is given in \cite{BdHN06}, Lemma 3.1.1. We repeat it here, because it uses 
basic properties of communication heights that provide useful insight.
 
\medskip\noindent
\underline{Upper bound}:
The upper bound is obtained from (\ref{capa}) by picking $h=1_{K(\cA,\cB)}$ with
\begin{equation}
\label{KABdef}
K(\cA,\cB) = \{\eta\in\cX \colon\,\Phi(\eta,\cA) \leq \Phi(\eta,\cB)\}.
\end{equation}
The key observation is that if $\eta \sim \eta'$ with $\eta \in K(\cA,\cB)$
and $\eta'\in \cX \backslash K(\cA,\cB)$, then
\begin{equation}
\label{outcond}
\begin{array}{lll}
&{\rm (1)} &H(\eta')<H(\eta),\\
&{\rm (2)} &H(\eta) \geq \Phi(\cA,\cB).
\end{array}
\end{equation}

To see (1), suppose that $H(\eta') \geq H(\eta)$. Clearly,
\begin{equation}
\label{imp1}
H(\eta') \geq H(\eta) \qquad \Longleftrightarrow\qquad 
\Phi(\eta',\cF)=\Phi(\eta,\cF) \vee H(\eta')
\,\,\,\forall\, \cF\subset\cX.
\end{equation}
But $\eta \in K(\cA,\cB)$ tells us that $\Phi(\eta,\cA) \leq \Phi(\eta,\cB)$, hence 
$\Phi(\eta',\cA) \leq \Phi(\eta',\cB)$ by (\ref{imp1}), and hence $\eta'\in K(\cA,\cB)$, 
which is a contradiction. 

To see (2), note that (1) implies the reverse of (\ref{imp1}):
\begin{equation}
\label{imp2}
H(\eta) \geq H(\eta') \qquad\Longleftrightarrow\qquad
\Phi(\eta,\cF)=\Phi(\eta',\cF) \vee H(\eta)
\,\,\,\forall\, \cF\subset\cX.
\end{equation}
Trivially, $\Phi(\eta,\cB)\geq H(\eta)$. We claim that equality holds. Indeed,
suppose that equality fails. Then we get
\begin{equation}
\label{imp3}
H(\eta) < \Phi(\eta,\cB) = \Phi(\eta',\cB) < \Phi(\eta',\cA) = \Phi(\eta,\cA),
\end{equation}
where the equalities come from (\ref{imp2}), while the second inequality uses 
that $\eta'\in \cX \backslash K(\cA,\cB)$. Thus, $\Phi(\eta,\cA)>\Phi(\eta,\cB)$,
which contradicts $\eta\in K(\cA,\cB)$. From $\Phi(\eta,\cB)=H(\eta)$ we obtain
$\Phi(\cA,\cB)\leq\Phi(\cA,\eta)\vee\Phi(\eta,\cB)=\Phi(\eta,\cB)=H(\eta)$, which 
proves (2).
  
Combining (\ref{outcond}) with (\ref{Gibbs}--\ref{rate}) and using reversibility, 
we find that
\begin{equation}
\label{muPhi}
\mu_\beta(\eta)c_\beta(\eta,\eta') 
\leq \frac{1}{Z_\beta}\, e^{-\beta\Phi(\cA,\cB)} \qquad \forall\,
\eta\in K(\cA,\cB),\,\eta'\in \cX\backslash K(\cA,\cB),\,\eta\sim\eta'.
\end{equation}
Hence
\begin{equation}
\label{hchoice1}
\CAPA_\beta(\cA,\cB) \leq \cE_\beta(1_{K(\cA,\cB)})
\leq C_2 \frac{1}{Z_\beta}\, e^{-\beta\Phi(\cA,\cB)}
\end{equation}
with $C_2 = |\{(\eta,\eta')\in\cX^2\colon\,\eta\in K(\cA,\cB),\eta'\in\cX\backslash
K(\cA,\cB),\eta\sim\eta'\}|$.

\medskip\noindent
\underline{Lower bound}:
The lower bound is obtained by picking any path $\omega=(\omega_0,\omega_1,\ldots,\omega_L)$ 
that realizes the minimax in $\Phi(\cA,\cB)$ and ignoring all the transitions that are not 
in this path, i.e.,
\begin{equation}
\label{path}
\CAPA_\beta(\cA,\cB) \geq \min_{{h\colon\omega\to [0,1]} \atop {h(\omega_0)=1,h(\omega_L)=0}} 
\cE^\omega_\beta(h),
\end{equation}
where the Dirichlet form $\cE^\omega_\beta$ is defined as $\cE_\beta$ in (\ref{Diri}) but 
with $\cX$ replaced by $\omega$. Due to the one-dimensional nature of the set $\omega$, 
the variational problem in the right-hand side can be solved explicitly by 
elementary computations. One finds that the minimum equals
\begin{equation}
\label{Mdef}
M = \left[\sum_{l=0}^{L-1} \frac{1}{\mu_\beta(\omega_l)c_\beta(\omega_l,\omega_{l+1})}\right]^{-1},
\end{equation}
and is uniquely attained at $h$ given by
\begin{equation}
\label{hminpath}
h(\omega_l) = M \sum_{k=0}^{l-1} \frac{1}{\mu_\beta(\omega_k)c_\beta(\omega_k,\omega_{k+1})},
\qquad l=0,1,\ldots,L.
\end{equation}
We thus have
\begin{equation}
\label{Mest}
\begin{aligned}
\CAPA_\beta(\cA,\cB) &\geq M\\ 
&\geq \frac{1}{L}\,\min_{l=0,1,\ldots,L-1} \mu_\beta(\omega_l)c_\beta(\omega_l,\omega_{l+1})\\
&= \frac{1}{K}\,\frac{1}{Z_\beta}\,\min_{l=0,1,\ldots,L-1}
e^{-\beta[H(\omega_l) \vee H(\omega_{l+1})]}\\
&= C_1\,\frac{1}{Z_\beta}\, e^{-\beta\Phi(\cA,\cB)}
\end{aligned}
\end{equation}
with $C_1=1/L$.
\end{proof}

%%%%%%%%%%%%%%%%%%%%%%%%%%%%%%

\subsection{Graph structure of the energy landscape}
\label{S2.2}

View $\cX$ as a graph whose \emph{vertices} are the configurations and whose \emph{edges}
connect communicating configurations, i.e., $(\eta,\eta')$ is an edge if and only if
$\eta\sim\eta'$. Define
\begin{itemize}
\item[--]
$\cX\starred$ is the subgraph of $\cX$ obtained by removing all vertices $\eta$ with
$H(\eta)>\Gamma\starred$ and all edges incident to these vertices;
\item[--]
$\cX\dstarred$ is the subgraph of $\cX\starred$ obtained by removing all vertices $\eta$ with
$H(\eta)=\Gamma\starred$ and all edges incident to these vertices;
\item[--]
$\cX_\Box$ and $\cX_\Boxplus$ are the connected components of $\cX\dstarred$ containing
$\Box$ and $\Boxplus$, respectively.
\end{itemize}

\begin{lemma}
\label{lemmavalleys}
The sets $\cX_\Box$ and $\cX_\Boxplus$ are disjoint (and hence are disconnected in 
$\cX\dstarred$), and
\begin{equation}
\label{valleys}
\begin{aligned}
\cX_\Box &= \{\eta\in\cX\colon\,\Phi(\eta,\Box) < \Phi(\eta,\Boxplus) = \Gamma\starred\},\\
\cX_\Boxplus &= \{\eta\in\cX\colon\,\Phi(\eta,\Boxplus) < \Phi(\eta,\Box) = \Gamma\starred\}.
\end{aligned}
\end{equation}
Moreover, $\proto \subset \cX_{\Box}$, and $\gate_\mathrm{att}(\hat{\eta})\cap \cX_\Boxplus 
\neq \emptyset$ for all $\hat{\eta}\in\proto$.
\end{lemma}

\begin{proof}
By Definition~\ref{def2}(f), all paths connecting $\Box$ and $\boxplus$ reach energy level 
$\geq \Gamma\starred$. Therefore $\cX_\Box$ and $\cX_\boxplus$ are disconnected in $\cX\dstarred$ 
(because $\cX\dstarred$ does not contain vertices with energy $\geq\Gamma\starred$). 

First note that, by (H2) and (\ref{ultramet}), $\Gamma\starred=\Phi(\Box,\Boxplus)\leq
\max\{\comlev(\eta,\Box),\comlev(\eta,\Boxplus)\} \leq \Gamma\starred$, and hence either
$\comlev(\eta,\Box)=\Gamma\starred$ or $\comlev(\eta,\Boxplus)=\Gamma\starred$ or both.
To check the first line of (\ref{valleys}) we argue as follows. For any $\eta\in\cX_{\Box}$, 
we have $H(\eta)<\Gamma\starred$ (because $\cX_\Box\subset\cX\dstarred$) and $\Phi(\eta,\Box)
<\Gamma\starred$ (because $\cX$ is connected). Conversely, let $\eta$ be such that 
$\Phi(\eta,\Box)<\Gamma\starred$. Then $H(\eta)<\Gamma\starred$, hence $\eta\in\cX\dstarred$, 
and there is a path connecting $\eta$ and $\Box$ that stays below energy level $\Gamma\starred$. 
Therefore $\eta$ belongs to the connected component of $\cX\dstarred$ containing $\Box$, i.e., 
$\eta\in\cX_{\Box}$. The second line of (\ref{valleys}) is checked in an analogous manner.

To prove that $\proto\subset\cX_{\Box}$, we must show that $\Phi(\Box,\hat\eta)<\Gamma\starred$
for all $\hat\eta \in \proto$. Pick any $\hat\eta\in\proto$, and let $\eta\in\entgate$ be any 
configuration obtained from $\hat\eta$ by adding a particle of type $\tb$ somewhere in 
$\partial^-\Lambda$. Denote by $\Omega(\eta)$ the set of optimal paths from $\Box$ to $\boxplus$
that enter $\cG(\Box,\boxplus)$ via $\eta$ (note that this set is non-empty because $\entgate$ 
is a minimal gate by Definition~\ref{defdroplets-a}(a)). By Definition~\ref{defdroplets-a}(b), 
$\omega_i \in \Omega(\eta)$ visits $\hat\eta$ before $\eta$ for all $i \in 1,\ldots,|\Omega(\eta)|$.
The proof proceeds via contradiction. Suppose that $\max_{\sigma\in\omega_i\backslash
S_i(\eta)} H(\sigma) = \Gamma\starred$ for all $i \in 1,\ldots,|\Omega(\eta)|$, where 
$S_i(\eta)$ consists of $\eta$ and all its successors in $\omega_i$. Let $\sigma\starred_i(\eta)$ 
be the last configuration $\sigma\in\omega_i\backslash S_i(\eta)$ such that 
$H(\sigma) = \Gamma\starred$, and put $\cL(\eta)=\{\sigma\starred_1(\eta),\ldots,
\sigma\starred_{|\Omega(\eta)|}(\eta)\}$. Then the set $(\entgate\backslash\eta)\cup
\cL(\eta)$ is a minimal gate. But $\omega_i$ hits $\sigma\starred_i(\eta)$ before 
$\eta$, and so this contradicts the fact that $\entgate$ is the entrance set of 
$\cG(\Box,\boxplus)$.

The claim that $\gate_\mathrm{att}(\hat{\eta})\cap \cX_\Boxplus \neq \emptyset$ for all 
$\hat{\eta}\in\proto$ is immediate from (H3-c). 
\end{proof}

We now have all the geometric ingredients that are necessary for the proof of 
Theorems~\ref{tgate}--\ref{texp} along the lines of \cite{BdHN06}, Section 3.
Our hypotheses (H1)--(H3) replace the somewhat delicate and model-dependent 
geometric analysis for Kawasaki dynamics with one type of particle that was carried 
out in \cite{BdHN06}, Section 2. They are the \emph{minimal hypotheses} that are 
necessary to carry out the proof below. Their verification for our specific model will be given in
\cite{dHNTpr1} and \cite{dHNTpr2}.

%%%%%%%%%%%%%%%%%%%%%%%%%%%

\subsection{Metastable set, link between average nucleation time and capacity}
\label{Smcnt}

Bovier, Eckhoff, Gayrard and Klein~\cite{BEGK02} define metastable sets in terms of
capacities:

\begin{definition}
\label{defmetapta}
$\cA\subset\cX$ with $\cA \neq \emptyset$ is called a metastable set if
\begin{equation}
\label{metpair}
\lim_{\beta\to\infty} \frac{
\max_{\eta\notin\cA} \mu_\beta(\eta)/\CAPA_\beta(\eta,\cA)
}{
\min_{\eta\in\cA} \mu_\beta(\eta)/\CAPA_\beta(\eta,\cA\backslash\eta)
} 
= 0.
\end{equation}
\end{definition}

\noindent
The following key lemma, relying on hypotheses (H1)--(H2) and Definition~\ref{def2}(d)--(e), 
allows us to apply the theory in \cite{BEGK02}.

\begin{lemma}
\label{lemcapameta}
$\{\Box,\Boxplus\}$ is a metastable set in the sense of Definition~{\rm \ref{defmetapta}}.
\end{lemma}

\begin{proof}
By \eqref{Gibbs}, Lemma~\ref{ltometapair} and the lower bound in \eqref{aprbds}, the
numerator is bounded from above by $e^{V^{\star}\beta}/C_1 = e^{(\Gamma\starred-\delta)
\beta}/C_1$ for some $\delta>0$. By \eqref{Gibbs}, the definition of $\Gamma\starred$ and
the upper bound in \eqref{aprbds}, the denominator is bounded from below by $e^{\Gamma
\starred\beta}/C_2$ (with the minimum being attained at $\Box$).
\end{proof}

Lemma~\ref{lemcapameta} has an important consequence:

\begin{lemma}
\label{lexpectation}
$\E_\Box(\tau_\Boxplus) = [Z_\beta\CAPA_\beta(\Box,\boxplus)]^{-1}\,[1+o(1)]$ as 
$\beta\to\infty$.
\end{lemma}

\begin{proof}
According to \cite{BEGK02}, Theorem 1.3(i), we have
\begin{equation}
\E_\Box(\tau_\Boxplus) 
= \frac{\mu_{\beta}(\cR_{\Box})}{\CAPA_{\beta}(\Box,\Boxplus)}\,[1+o(1)]
\qquad \text{as } \beta \to \infty,
\end{equation}
where
\begin{equation}
\label{RBoxdef}
\cR_{\Box} = \big\{\eta\in\cX\colon\, 
\P_{\eta} (\tau_{\Box}<\tau_{\Boxplus}) \geq \P_{\eta}(\tau_{\Boxplus}<\tau_{\Box})\big\}.
\end{equation}
Recalling (\ref{h*def}), we can rewrite (\ref{RBoxdef}) as $\cR_{\Box}=\{\eta\in\cX\colon\,
h\starred_{\Box,\Boxplus}(\eta)\geq\tfrac12\}$. It follows from Lemma~\ref{lh*est} below that
\begin{equation}
\lim_{\b\to\infty} \min_{\eta\in\cX_\Box} h\starred_{\Box,\Boxplus}(\eta) = 1,
\qquad \lim_{\b\to\infty} \max_{\eta\in\cX_\Boxplus} h\starred_{\Box,\Boxplus}(\eta) = 0.
\end{equation}
Hence, for $\beta$ large enough,
\begin{equation}
\cX_{\Box}\subset\cR_{\Box}\subset\cX\backslash\cX_{\Boxplus}.
\end{equation}
By Lemma~\ref{lemmavalleys}, the second inclusion implies that $\Phi(\eta,\Box)\leq
\Phi(\eta,\Boxplus)$ for all $\eta\in\cR_{\Box}$. Therefore Lemma~\ref{lemptyisbottom} 
yields
\begin{equation}
\min_{\eta\in\cR_{\Box}\backslash\Box} H(\eta)>H(\Box)=0,
\end{equation}
which implies that $\mu_{\beta}(\cR_{\Box})/\mu_{\beta}(\Box) = 1+o(1)$. Since $\mu_{\beta}(\Box) 
= 1/Z_{\beta}$, the claim follows.
\end{proof}

Lemma~\ref{lexpectation} shows that the proof of Theorem~\ref{tnucltime} revolves around 
getting sharp bounds on $Z_\beta\CAPA_\beta(\Box,\Boxplus)$. The a priori estimates in 
Lemma \ref{lemmaprbds} serve as a jump board for the derivation of these bounds.

%%%%%%%%%%%%%%%%%%%%%%%%%%%%%%%%%%%%%%%

\subsection{Proof of Theorem~\ref{tnucltime}}
\label{S2.3}

Our starting point is Lemma~\ref{lexpectation}. Recalling (\ref{Diri}--\ref{h*def}),
our task is to show that
\begin{equation}
\label{Diricut}
\begin{aligned}
Z_\beta\CAPA_\beta(\Box,\Boxplus)
&= \tfrac12 \sum_{\eta,\eta'\in\cX} Z_\beta\mu_\beta(\eta)c_\beta(\eta,\eta')\,
[h\starred_{\Box,\Boxplus}(\eta)-h\starred_{\Box,\Boxplus}(\eta')]^2\\
&= [1+o(1)]\,\Theta\,e^{-\Gamma\starred\beta} \qquad \mbox{ as } \beta\to\infty,
\end{aligned}
\end{equation}
and to identify the constant $\Theta$, since \eqref{Diricut} will imply
\eqref{sharpasymp} with $\Theta=1/K$. This is done in four steps, organized
in Sections~\ref{Step1}--\ref{Step4}.

%%%%%%%%%%%

\subsubsection{Step 1: Triviality of $h\starred_{\Box,\Boxplus}$ on $\cX_\Box$,
$\cX_\Boxplus$ and $\cX\dstarred\backslash(\cX_\Box\cup\cX_\Boxplus)$}
\label{Step1}

For all $\eta\in\cX\backslash\cX\starred$ we have $H(\eta)>\Gamma\starred$, and so there exists
a $\delta>0$ such that $Z_\beta\mu_\beta(\eta) \leq e^{-(\Gamma\starred+\delta)\beta}$.
Therefore, we can replace $\cX$ by $\cX\starred$ in the sum in \eqref{Diricut} at the
cost of a prefactor $1+O(e^{-\delta\beta})$. Moreover, we have the following analogue 
of \cite{BdHN06}, Lemma 3.3.1.

\begin{lemma}
\label{lh*est}
There exist $C<\infty$ and $\delta>0$ such that
\begin{equation}
\label{h*triv}
\min_{\eta\in\cX_\Box} h\starred_{\Box,\Boxplus}(\eta) \geq 1-Ce^{-\delta\beta},
\qquad
\max_{\eta\in\cX_\Boxplus} h\starred_{\Box,\Boxplus}(\eta) \leq Ce^{-\delta\beta},
\qquad \forall\,\beta\in (0,\infty).
\end{equation}
\end{lemma}

\begin{proof}
A standard renewal argument gives the relations, valid for $\eta\notin\{\Box,\boxplus\}$,
\be{renew}
\P_\eta(\t_\boxplus<\t_\Box) = \frac{\P_\eta(\t_\boxplus<\t_{\Box\cup\eta})}
{1-\P_\eta(\t_{\Box\cup\boxplus}>\t_\eta)}, \qquad
\P_\eta(\t_\Box<\t_\boxplus) = \frac{\P_\eta(\t_\Box<\t_{\boxplus\cup\eta})}
{1-\P_\eta(\t_{\Box\cup\boxplus}>\t_\eta)}.
\ee
	
For $\eta\in\cX_\Box\backslash\Box$, we estimate
\be{est1}
h_{\Box,\boxplus}\starred(\eta)
= 1-\P_\eta(\t_\boxplus<\t_\Box)
= 1-\frac{\P_\eta(\t_\boxplus<\t_{\Box\cup\eta})}
{\P_\eta(\t_{\Box\cup\boxplus}<\t_\eta)}
\geq 1-\frac{\P_\eta(\t_\boxplus<\t_\eta)}
{\P_\eta(\t_\Box<\t_\eta)}
\ee
and, with the help of (\ref{caprep}) and Lemma \ref{lemmaprbds},
\be{est1a}
\frac{\P_\eta(\t_\boxplus<\t_\eta)}{\P_\eta(\t_\Box<\t_\eta)}
= \frac{Z_\b\,\CAPA_\b(\eta,\boxplus)}{Z_\b\,\CAPA_\b(\eta,\Box)}
\leq C(\eta)\,e^{-[\Phi(\eta,\boxplus)-\Phi(\eta,\Box)]\b}
\leq C(\eta)\,e^{-\d\b},
\ee
which proves the first claim with $C=\max_{\eta\in\cX_\Box\backslash\Box} C(\eta)$. Note that 
$h\starred_{\Box,\boxplus}(\Box)$ is a convex combination of $h_{\Box,\boxplus}\starred(\eta)$ with 
$\eta\in\cX_\Box\backslash\Box$, and so the claim includes $\eta=\Box$.
	
For $\eta\in\cX_\boxplus\backslash\boxplus$, we estimate
\be{est2}
h_{\Box,\boxplus}\starred(\eta)
= \P_\eta(\t_\Box<\t_\boxplus)
= \frac{\P_\eta(\t_\Box<\t_{\boxplus\cup\eta})}
{\P_\eta(\t_{\Box\cup\boxplus}<\t_\eta)}
\leq \frac{\P_\eta(\t_\Box<\t_\eta)}
{\P_\eta(\t_\boxplus<\t_\eta)}
\ee
and, with the help of (\ref{caprep}) and Lemma \ref{lemmaprbds},
\be{est2a}
\frac{\P_\eta(\t_\Box<\t_\eta)}{\P_\eta(\t_\boxplus<\t_\eta)}
= \frac{Z_\b\,\CAPA_\b(\eta,\Box)}{Z_\b\,\CAPA_\b(\eta,\boxplus)}
\leq C(\eta)\,e^{-[\Phi(\eta,\Box)-\Phi(\eta,\boxplus)]\b}
\leq C(\eta)\,e^{-\d\b},
\ee
which proves the second claim with $C=\max_{\eta\in\cX_\boxplus\backslash\boxplus}C(\eta)$.
\end{proof}

In view of Lemma~\ref{lh*est}, $h\starred_{\Box,\Boxplus}$ is trivial on the set 
$\cX_\Box\cup\cX_\Boxplus$, and its contribution to the sum in \eqref{Diricut}, 
which is $O(e^{-\d\b})$, can be accounted for by the prefactor $1+o(1)$. Consequently, 
all that is needed is to understand what $h\starred_{\Box,\Boxplus}$ looks like on the set
\begin{equation}
\label{setsa}
\cX\starred\backslash(\cX_\Box\cup\cX_\Boxplus) = \{\eta\in\cX\starred\colon\,
\Phi(\eta,\Box) = \Phi(\eta,\Boxplus) = \Gamma\starred\}.
\end{equation}
However, $h\starred_{\Box,\Boxplus}$ is also trivial on the set
\begin{equation}
\label{setsb}
\cX\dstarred\backslash(\cX_\Box\cup\cX_\Boxplus) = \bigcup_{i=1}^I \cX_i,
\end{equation}
which is a union of wells $\cX_i$, $i=1,\dots,I$, in $\cS(\Box,\Boxplus)$ for some
$I\in\N$. (Each $\cX_i$ is a maximal set of communicating configurations with energy 
$<\Gamma\starred$ and with communication height $\Gamma\starred$ towards both $\Box$ and
$\Boxplus$.) Namely, we have the following analogue of \cite{BdHN06}, Lemma 3.3.2.

\begin{lemma}
\label{lh*trivialinXi}
There exist $C<\infty$ and $\d>0$ such that
\begin{equation}
\label{h*trivalt}
\max_{\eta,\eta'\in\cX_{i}} |h\starred_{\Box,\Boxplus}(\eta)-h\starred_{\Box,\Boxplus}(\eta')| 
\leq Ce^{-\delta\beta} \qquad \forall\,i = 1,\ldots,I,\,\beta\in (0,\infty).
\end{equation}
\end{lemma}

\begin{proof}
Fix $i$. Let $\h' \in \cX_{i}$ be such that $\min_{\sigma\in\cX_{i}} H(\sigma)= H(\eta_{i})$ 
and pick $\h \in \cX_{i}$. Estimate
\be{Bd1}
h\starred_{\Box,\boxplus}(\h) = \P_\h(\t_\Box<\t_\boxplus)
\leq \P_\h(\t_\Box<\t_{\h'}) + \P_\h(\t_{\h'}<\t_\Box<\t_\boxplus).
\ee
First, as in the proof of Lemma \ref{lh*est}, we have 
\be{Bd2}
\begin{aligned}
\P_\h(\t_\Box<\t_{\h'}) &= \frac{\P_\h(\t_\Box<\t_{\h\cup\h'})}
{1-\P_\h(\t_{\Box\cup\h'}>\t_\h)}
\leq \frac{\P_\h(\t_\Box<\t_\h)}
{\P_\h(\t_{\h'}<\t_\h)}\\
&=\frac{Z_\b\CAPA_\b(\h,\Box)}{Z_\b\CAPA_\b(\h,\h')}
\leq C(\h,\h')\,e^{-[\Phi(\h,\Box)-\Phi(\h,\h')]\b}
\leq C(\h,\h')\,e^{-\d\b},
\end{aligned}
\ee
where we use that $\Phi(\h,\Box) = \Gamma\starred$ and $\Phi(\h,\h')<\Gamma\starred$. Second,
\be{Bd3}
\P_\h(\t_{\h'}<\t_\Box<\t_\boxplus) 
= \P_\h(\t_{\h'}<\t_{\Box\cup\boxplus})
\P_{\h'}(\t_\Box<\t_\boxplus)
\leq \P_{\h'}(\t_\Box<\t_\boxplus)
= h\starred_{\Box,\boxplus}(\h').
\ee
Combining (\ref{Bd1}--\ref{Bd3}), we get
\be{Bd4}
h\starred_{\Box,\boxplus}(\h) \leq C(\h,\h')\,e^{-\d\b} 
+ h\starred_{\Box,\boxplus}(\h').
\ee
Interchanging $\eta$ and $\eta'$,we get the claim with $C=\max_i\max_{\h,\h'\in\cX_i} C(\h,\h')$.
\end{proof}

In view of Lemma~\ref{lh*trivialinXi}, the contribution to the sum in \eqref{Diricut} of the 
transitions inside a well can also be put into the prefactor $1+o(1)$. Thus, only the transitions 
\emph{in and out of wells} contribute.

%%%%%%%%%%%%%%

\subsubsection{Step 2: Variational formula for $K$}
\label{Step2} 

By Step 1, the estimation of $Z_\beta\CAPA_\beta(\Box,\Boxplus)$ reduces to the study 
of a simpler variational problem. The following is the analogue of \cite{BdHN06}, Proposition 3.3.3. 

\begin{lemma}
\label{lreduction}
$Z_\beta\CAPA_\beta(\Box,\Boxplus) = [1+o(1)]\,\Theta\,e^{-\Gamma\starred\beta}$ as $\beta\to\infty$
with
\begin{equation}
\label{redvp}
\Theta = 
\min_{C_1\ldots,C_I} \min_{ {h\colon\,\cX\starred \to [0,1]} 
\atop {h|_{\cX_\Box} \equiv 1,\,h|_{\cX_\Boxplus} \equiv 0,
\,h|_{\cX_i} \equiv C_i\,\forall\,i=1,\ldots,I} }
\qquad \tfrac12 \sum_{\eta,\eta'\in\cX\starred}
1_{\{\eta\sim\eta'\}}\, [h(\eta)-h(\eta')]^2.
\end{equation}
\end{lemma}

\begin{proof}
First, recalling (\ref{Gibbs}--\ref{rate}) and (\ref{Diri}--\ref{capa}), we have
\be{red1a}
\begin{aligned}
Z_\b\,\CAPA_\b(\Box,\boxplus) &= Z_\b\,\min_{{h\colon\,\cX\to[0,1]} 
\atop {h(\Box)=1,\,h(\boxplus)=0}} 
\tfrac12 \sum_{\h,\h'\in\cX}
\mu_\b(\h)c_\b(\h,\h')[h(\h)-h(\h')]^2\\
&=O\left(e^{-(\Gamma\starred+\d)\b}\right)  
+ Z_\b\,\min_{{h\colon\,\cX\starred\to[0,1]} 
\atop {h(\Box)=1,\,h(\boxplus)=0}} 
\tfrac12 \sum_{\h,\h'\in\cX\starred}
\mu_\b(\h)c_\b(\h,\h')[h(\h)-h(\h')]^2.
\end{aligned}
\ee
Next, with the help of Lemmas \ref{lh*est}--\ref{lh*trivialinXi}, we get 
\be{red1b}
\begin{aligned}
&\min_{{h\colon\,\cX\starred\to[0,1]} \atop {h(\Box)=1,\,h(\boxplus)=0}} 
\tfrac12 \sum_{\h,\h'\in\cX\starred}
\mu_\b(\h)c_\b(\h,\h')[h(\h)-h(\h')]^2\\
&= \min_{{h\colon\,\cX\starred\to[0,1]} \atop {h=h\starred_{\Box,\boxplus}
\mbox{ on } \cX_\Box\cup\cX_\boxplus\cup(\cX_1,\dots,\cX_I)}} 
\tfrac12 \sum_{\h,\h'\in\cX\starred}
\mu_\b(\h)c_\b(\h,\h')[h(\h)-h(\h')]^2\\
&= [1+O(e^{-\d\b})]\,\min_{C_1,\dots,C_I}
\min_{{h\colon\,\cX\starred\to[0,1]} \atop {h|_{\cX_\Box}\equiv 1,\,
h|_{\cX_\boxplus}\equiv 0,\,h|_{\cX_i}\equiv C_i\,\forall\,i=1,\dots,I}}
\tfrac12 \sum_{\h,\h'\in\cX\starred}
\mu_\b(\h)c_\b(\h,\h')[h(\h)-h(\h')]^2,
\end{aligned}
\ee
where the error term $O(e^{-\d\b})$ arises after we replace the \emph{approximate} 
boundary conditions
\be{bcapprox1}
h=\left\{\begin{array}{ll}
1-O(e^{-\d\b})    &\mbox{on } \cX_\Box,\\
O(e^{-\d\b})      &\mbox{on } \cX_\boxplus,\\
C_i+O(e^{-\d\b})  &\mbox{on } \cX_i,\,i=1,\dots,I,
\end{array}
\right.
\ee
by the \emph{sharp} boundary conditions
\be{bcapprox2}
h=\left\{\begin{array}{ll}
1     &\mbox{on } \cX_\Box,\\
0     &\mbox{on } \cX_\boxplus,\\
C_i   &\mbox{on } \cX_i,\,i=1,\dots,I.
\end{array}
\right.
\ee
Finally, by (\ref{Gibbs}--\ref{rate}) and reversibility, we have
\be{red1c}
\begin{array}{ll}
&Z_\b\mu_\b(\h)c_\b(\h,\h') = 1_{\{\h\sim\h'\}}\, e^{-\Gamma\starred\b}
\mbox{ for all } \h,\h'\in\cX\starred \mbox{ that are not either}\\
&\mbox{both in } \cX_\Box \mbox{ or both in } \cX_\boxplus 
\mbox{ or both in } \cX_i \mbox{ for some } i=1,\dots,I.
\end{array}
\ee
To check the latter, note that there are no allowed moves between these sets, so that
either $H(\eta)=\Gamma\starred>H(\eta')$ or $H(\eta)<\Gamma\starred=H(\eta')$ for allowed 
moves in and out of these sets.
\end{proof}

Combining Lemmas~\ref{lexpectation} and \ref{lreduction}, we see that we have completed 
the proof of \eqref{sharpasymp} with $K=1/\Theta$. The variational formula for $\Theta
=\Theta(\Lambda;U,\Da,\Db)$ is \emph{non-trivial} because it depends on the geometry of 
the wells $\cX_i$, $i=1,\ldots,I$.

%%%%%%%%%%%%%%

\subsubsection{Step 3: Bounds on $K$ in terms of capacities of simple random walk}
\label{Step3}

So far we have \emph{only} used (H1)--(H2). In the remainder of the proof we use (H3) to
prove (\ref{Kasymp}). The intuition behind (\ref{Kasymp}) is the following. When the 
free particle attaches itself to the protocritial droplet, the dynamics enters the 
set $\gate_\mathrm{att}$. The entrance configurations of $\gate_\mathrm{att}$ are 
either in $\cX_\Boxplus$ or in one of the $\cX_i$'s. In the former case the path can 
reach $\Boxplus$ while staying below $\Gamma\starred$ in energy, in the latter case 
it cannot. By Lemma~\ref{lbacktrack}, if the path exits an $\cX_i$, then for it to 
return to $\cX_\Box$ it must pass through $\entgate$, i.e., it must go through a 
series of configurations consisting of a single protocritical droplet and a free 
particle moving away from that protocritical droplet towards $\partial^-\Lambda$. 
Now, this backward motion has a small probability because simple random walk in 
$\Z^2$ is \emph{recurrent}, namely, the probability is $[1+o(1)]\,4\pi/\log|\Lambda|$ 
as $\Lambda \to\Z^2$ (see \cite{BdHN06}, Equation (3.4.5)). Therefore, the free 
particle is likely to re-attach itself to the protocritical droplet before it manages 
to reach $\partial^-\Lambda$. Consequently, with a probability tending to 1 as 
$\Lambda\to\Z^2$, before the free particle manages to reach $\partial^-\Lambda$ it 
will re-attach itself to the protocritical droplet in all possible ways, which must 
include a way such that the dynamics enters $\cX_\Boxplus$. In other words, after 
entering $\gate_\mathrm{att}$ the path is likely to reach $\cX_\Boxplus$ before it 
returns to $\cX_\Box$, i.e., it ``goes over the hill''. Thus, in the limit as 
$\Lambda\to\Z^2$, the $\cX_i$'s become \emph{irrelevant}, and the dominant role is 
played by the transitions in and out of $\cX_\square$ and by the simple random walk 
performed by the free particle.

\br{remXi}
{\rm The protocritical droplet may change each time the path enters and exits an 
$\cX_i$. There are $\cX_i$'s from which the path can reach $\Boxplus$ without going 
back to $\gate$ and without exceeding $\Gamma\starred$ in energy (see the proof of
\cite{BdHN06}, Theorem 1.4.3, where this is shown for Kawasaki dynamics with one 
type of particle).}
\er

In order to make the above intuition precise, we need some further notation.

\begin{definition}
\label{extnot}
(a) For $F\subset\Z^2$, $\partial^+F$ and $\partial^-F$ are the external, respectively, 
internal boundary of $F$.\\
(b) For $\eta\in\cX$, $\supp(\eta)$ is the set of occupied sites of $\eta$.\\
(c) For $\eta\in\gate\cup\gate_\mathrm{att}$, write $\eta = (\hat{\eta},x)$ with 
$\hat{\eta}\in\proto$ the protocritical droplet and $x\in\Lambda$ the location 
of the free/attached particle of type $\tb$.\\
(d) For $\hat{\eta}\in\proto$, $A(\hat{\eta}) = \{x\in\partial^+\supp(\hat{\eta})
\colon\, H(\hat{\eta},x)<\Gamma\starred\}$ is the set of sites where the free particle 
of type $\tb$ can attach itself to a particle of type $\ta$ in $\partial^-\supp(\eta)$
to form an active bond. Note that $x \in A(\hat{\eta})$ if and only if $\eta = (\hat{\eta},x) 
\in \gate_\mathrm{att}$, and that for every $\eta\in \gate_\mathrm{att}$ either 
$\eta\in\cX_{\boxplus}$ or $\eta\in\cX_i$ for some $i = 1,\ldots,I$.\\
(e) For $\hat{\eta}\in\proto$, let
\begin{equation}
\label{gbsites}
\begin{aligned}
G(\hat{\eta}) &= \{x\in A(\hat{\eta})\colon\,(\hat{\eta},x)\in\cX_{\boxplus}\},\\
B(\hat{\eta}) &= \{x\in A(\hat{\eta})\colon\,\exists\, i=1,\ldots,I\colon\,(\hat{\eta},x)
\in\cX_i\},
\end{aligned}
\end{equation} 
be called the set of good sites, respectively, bad sites. Note that $(\hat{\eta},x)$ may be 
in the same $\cX_i$ for different $x\in B(\hat{\eta})$.\\
(f) For $\hat{\eta}\in\proto$, let 
\begin{equation}
\label{cardbsites}
I(\hat{\eta}) = \{i\in 1,\ldots,I\colon\,\exists\,
x\in B(\hat{\eta})\colon\,(\hat{\eta},x)\in\cX_i\}.
\end{equation} 
Note that $B(\hat{\eta})$ can be partitioned into disjoint sets $B_{1}(\hat{\eta}),\dots,
B_{|I(\hat{\eta})|}(\hat{\eta})$ according to which $\cX_i$ the configuration $(\hat{\eta},x)$ 
belongs to.\\
(g) Write $\CS(\hat{\eta}) = \supp(\hat{\eta}) \cup G(\hat{\eta})$, $\CS^+(\hat{\eta})
=\partial^+\CS(\hat{\eta})$ and $\CS^{++}(\hat{\eta})=\partial^+\CS^+(\hat{\eta})$.
\end{definition}

\noindent
Note that Definitions~\ref{extnot}(c--d) rely on (H3-a), and that $G(\hat{\eta})\neq\emptyset$
for all $\hat{\eta}\in\proto$ by (H3-c) and Lemma~\ref{lemmavalleys}. For the argument below
it is of no relevance whether $B(\hat{\eta})\neq\emptyset$ for some or all $\hat{\eta}\in\proto$.

The following lemma is the analogue of \cite{BdHN06}, Proposition 3.3.4.

\begin{lemma}
\label{red2}
$\Theta \in [\Theta_1,\Theta_2]$ with
\be{Theta12d=2}
\begin{aligned}
\Theta_1 &= [1+o(1)] \sum_{\hat{\eta}\in\proto} \CAPA^{\,\Lambda^+} 
\left(\partial^+\Lambda,\CS(\hat{\eta})\right),\\ 
\Theta_2 &= \sum_{\hat\eta\in\proto} \CAPA^{\,\Lambda^+}
\left(\partial^+\Lambda,\CS^{++}(\hat{\eta})\right),
\end{aligned}
\ee
where
\be {red-3d.2d=2}
\CAPA^{\,\Lambda^+} \left(\partial^+\Lambda,F\right)
= \min_{{g\colon\,\Lambda^+\to [0,1]}\atop
{g|_{\partial^+\Lambda}\equiv 1,\,g|_F\equiv 0}}
\tfrac 12 \sum_{(x,x')\in(\Lambda^+)\starred} [g(x)-g(x')]^2,
\qquad F\subset\Lambda,
\ee
with $(\Lambda^+)\starred=\{(x,y)\colon\,x,y\in\Lambda^+,|x-y|=1\}$, 
and $o(1)$ an error term that tends to zero as $\Lambda\to\Z^2$.
\end{lemma}

\bpr
The variational problem in (\ref{redvp}) decomposes into disjoint
variational problems for the maximally connected components of $\cX\starred$. 
Only those components that contain $\cX_\Box$ or $\cX_\boxplus$ 
contribute, since for the other components the minimum is achieved by 
picking $h$ constant.
	
\medskip\noindent
\underline{$\Theta \geq \Theta_1$}:
A lower bound is obtained from (\ref{redvp}) by removing all 
transitions that do not involve a fixed protocritical droplet 
and a move of the free/attached particle of type $\tb$. This 
removal gives
\be{Theta1bd}
\begin{aligned}
&\Theta \geq \sum_{\hat{\eta}\in\proto}\,\,
\min_{C_i(\hat{\eta}),\,i \in I(\hat{\eta})}\,\,
\min_{{g\colon\Lambda^+ \to [0,1]} \atop {g|_{G(\hat{\eta})}\equiv 0,\,
g|_{B_i(\hat{\eta})} \equiv C_i(\hat{\eta}),\,i \in I(\hat{\eta}),\,
g|_{\partial^+\Lambda} \equiv 1}}\\ 
&\qquad\qquad\qquad 
\tfrac12 \sum_{(x,x') \in [\Lambda^+\backslash\supp(\hat{\eta})]\starred} [g(x)-g(x')]^2.
\end{aligned}
\ee
To see how this bound arises from (\ref{redvp}), pick $h$ in \eqref{redvp} and 
$g$ in \eqref{Theta1bd} such that
\be{hchoicered}
h(\h) = h(\hat{\eta},x)=g(x), \qquad 
\hat{\eta}\in\proto,\,x\in\Lambda^+\backslash\supp(\hat{\eta}),
\ee 
and use that, by Definitions~\ref{extnot}(c--f), for every $\hat{\eta}\in\proto$ 
(recall Lemma~\ref{lemmavalleys})	
\begin{equation}
\label{lbprelim}
\begin{array}{lll}
&(\hat{\eta},x) \in \cX_{\boxplus},  &x \in G(\hat{\eta}),\\
&(\hat{\eta},x) \in \cX_i &x \in B_i(\hat{\eta}),\,i \in I(\hat{\eta}),\\ 
&(\hat{\eta},x) \in \proto \subset \cX_{\Box}, &x\in\partial^+\Lambda.
\end{array}
\end{equation} 
A further lower bound is obtained by removing from the right-hand side of 
\eqref{lbprelim} the boundary condition on the sets $B_i(\hat{\eta})$, $i \in I(\hat{\eta})$. 
This gives 
\be{Theta1bdalt}
\begin{aligned}
\Theta &\geq \sum_{\hat{\eta}\in\proto}\,\,
\min_{{g\colon\Lambda^+ \to [0,1]} \atop {g|_{G(\hat{\eta})}\equiv 0,\,
g|_{\partial^+\Lambda}\equiv 1}}
\tfrac12 \sum_{(x,x') \in [\Lambda^+\backslash\supp(\hat{\eta})]\starred} [g(x)-g(x')]^2\\
&= \sum_{\hat{\eta}\in\proto} \CAPA^{\,\Lambda^+\backslash\supp(\hat{\eta})}
\left(\partial^+\Lambda,G(\hat{\eta})\right),
\end{aligned}
\ee
where the upper index $\Lambda^+\backslash\supp(\hat{\eta})$ refers to the fact that no moves 
in and out of $\supp(\hat{\eta})$ are allowed (i.e., this set acts as an obstacle for the 
free particle). To complete the proof we show that, in the limit as $\Lambda\to\Z^2$,
\begin{equation}
\label{capawawoo}
\begin{aligned}
&\CAPA^{\,\Lambda^+}\left(\partial^+\Lambda,\supp(\hat{\eta}) \cup G(\hat{\eta})\right)
\geq \CAPA^{\,\Lambda^+\backslash\supp(\hat{\eta})}
\left(\partial^+\Lambda,G(\hat{\eta})\right)\\
&\geq \CAPA^{\,\Lambda^+}\left(\partial^+\Lambda,\supp(\hat{\eta}) \cup G(\hat{\eta})\right)
- O([1/\log|\Lambda|]^2).
\end{aligned}
\end{equation}
Since $\CS(\hat{\eta})=\supp(\hat{\eta}) \cup G(\hat{\eta})$ and, as we will show in Step 4 
below, $\CAPA^{\,\Lambda^+}(\partial^+\Lambda,\CS(\hat{\eta}))$ decays like $1/\log|\Lambda|$, 
the lower bound follows. 

Before we prove \eqref{capawawoo}, note that the capacity in the right-hand side of 
\eqref{capawawoo} includes more transitions than the capacity in the left-hand side, namely, 
all transitions from $\supp(\hat{\eta})$ to $B(\hat{\eta})$. Let 
\begin{equation}
g^{\Lambda^+\backslash\supp(\hat{\eta})}_{\partial^+\Lambda,G(\hat{\eta})}(x) 
= \mbox{ equilibrium potential for }
\CAPA^{\,\Lambda^+\backslash\supp(\hat{\eta})}\left(\partial^+\Lambda,G(\hat{\eta})\right)
\mbox{ at } x.
\end{equation}
Below we will show that $g^{\Lambda^+\backslash\supp(\hat{\eta})}_{\partial^+\Lambda,
G(\hat{\eta})}(x) \leq C/\log|\Lambda|$ for all $x\in B(\hat{\eta})$ and some $C<\infty$. 
Since in the Dirichlet form in (\ref{red-3d.2d=2}) the equilibrium potential appears 
squared, the error made by adding to the capacity in the left-hand side of (\ref{capawawoo}) 
the transitions from $\supp(\hat{\eta})$ to $B(\hat{\eta})$ therefore is of order 
$[1/\log|\Lambda|]^{2}$ times $|B(\hat{\eta})|$, which explains how (\ref{capawawoo}) 
arises.

Formally, let $\P^{\hat{\eta}}_{x}$ be the law of the simple random walk that starts at 
$x \in B(\hat{\eta})$ and is forbidden to visit the sites in $\supp(\hat{\eta})$. Let 
$y \in G(\hat{\eta})$. Using a renewal argument similar to the one used in the proof 
of Lemma~\ref{lh*est}, and recalling the probabilistic interpretation of the equilibrium 
potential in (\ref{h*def}) and of the capacity in (\ref{caprep}), we get 
\begin{equation}
\label{renewcapa}
\begin{aligned}
g^{\Lambda^+\backslash\supp(\hat{\eta})}_{\partial^+\Lambda,G(\hat{\eta})}(x) 
& = \P^{\hat{\eta}}_{x} (\tau_{\partial^{+}\Lambda} < \tau_{G(\hat{\eta})}) 
= \frac	{\P^{\hat{\eta}}_{x}(\tau_{\partial^{+}\Lambda} < \tau_{G(\hat{\eta}) \cup x })}
{\P^{\hat{\eta}}_{x}(\tau_{G(\hat{\eta}) \cup \partial^{+}\Lambda} < \tau_{x})} \\	
& \leq \frac{\P^{\hat{\eta}}_{x}(\tau_{\partial^{+}\Lambda} < \tau_{x})}
{\P^{\hat{\eta}}_{x}(\tau_{y} < \tau_{x})}
= \frac{\CAPA^{\,\Lambda^+\backslash\supp(\hat{\eta})}\left(x,\partial^+\Lambda\right)}
{\CAPA^{\,\Lambda^+\backslash\supp(\hat{\eta})}\left(x,y\right)}.
\end{aligned}
\end{equation}
The denominator of (\ref{renewcapa}) can be bounded from below by some $C'>0$ 
that is independent of $x$, $y$ and $\supp(\hat{\eta})$. To see why, pick a path 
from $x$ to $y$ that avoids $\supp(\hat{\eta})$ but stays inside an $L\starred
\times L\starred$ square around $\hat{\eta}$ (recall (H3-a)), and argue as in the 
proof of the lower bound of Lemma~\ref{lemmaprbds}. On the other hand, the numerator 
is bounded from above by $\CAPA^{\,\Lambda^+}(\partial^+\Lambda,G(\hat{\eta}))$, i.e., 
by the capacity of the same sets for a random walk that is not forbidden to visit 
$\supp(\hat{\eta})$, since the Dirichlet problem associated to the latter has the 
same boundary conditions, but includes more transitions. In the proof of Lemma~\ref{capasymp} 
below, we will see that $\CAPA^{\,\Lambda^+}(\partial^+\Lambda,G(\hat{\eta}))$ decays 
like $C''/\log|\Lambda|$ for some $C''<\infty$ (see (\ref{capextrid}--\ref{LB3}) below). 
We therefore conclude that indeed $g^{\supp(\hat{\eta})}_{\partial^+\Lambda,G(\hat{\eta})}(x) 
\leq C/\log|\Lambda|$ for all $x\in B(\hat{\eta})$ with $C = C''/C'$. 

\medskip\noindent
\underline{$\Theta \leq \Theta_2$}:
The upper bound is obtained from (\ref{redvp}) by picking $C_i=0$, $i=1,\ldots,I$, and
\be{hchoicered*}
h(\h) = \left\{\begin{array}{ll}
1      &\mbox{for } \h\in\cX_\Box,\\
g(x)   &\mbox{for } \h=(\hat\h,x)\in\cC^{++},\\
0      &\mbox{for } \h\in\cX\starred\backslash[\cX_\Box\cup\cC^{++}], 
\end{array}
\right.
\ee
where
\be{C++}
\cC^{++} = \big\{\h=(\hat{\eta},x)\colon\,\hat{\eta}\in\proto,
\,x \in \Lambda \backslash \CS^{++}(\hat{\eta})\big\}
\ee
consists of those configurations in $\gate$ for which the free particle is at distance 
$\geq 2$ of the protocritical droplet. The choice in \eqref{hchoicered*} gives
\be{Theta2bd}
\Theta \leq \sum_{\hat{\eta}\in\proto} \CAPA^{\,\Lambda^+} 
\left(\partial^+\Lambda,\CS^{++}(\hat{\eta})\right).
\end{equation}
To see how this upper bound arises, note that:
\begin{itemize}
\item
The choice in \eqref{hchoicered*} satisfies the boundary conditions in (\ref{redvp}) because
(recall (\ref{setsa}--\ref{setsb}))
\be{subC*}
\cC^{++} \subset \gate,\,\,[\cX_\Box\cup\gate] \cap [\cX_\boxplus \cup (\cup_{i=1}^I \cX_i)] 
= \emptyset \quad \Longrightarrow \quad \cX\starred \backslash [\cX_\Box\cup\cC^{++}] \supset 
[\cX_\boxplus \cup (\cup_{i=1}^I \cX_i)].
\ee
\item
By Lemma~\ref{lemmavalleys}, $\proto\subset\cX_\Box$. Therefore the first line of 
\eqref{hchoicered*} implies that $h(\h)=1$ for $\h=(\hat{\eta},x)$ with $\hat{\eta}\in\proto$ 
and $x\in\partial^+\Lambda$, which is consistent with the boundary condition $g|_{\partial^+\Lambda}
\equiv 1$ in (\ref{red-3d.2d=2}).
\item
The third line of \eqref{hchoicered*} implies that $h(\h)=0$ for $\h=(\hat{\eta},x)$ with 
$\hat{\eta}\in\proto$ and $x\in\CS^{++}(\hat{\eta})$, which is consistent with the boundary 
condition $g|_{F} \equiv 0$ in (\ref{red-3d.2d=2}) for $F=\CS^{++}(\hat{\eta})$. 
\end{itemize}
Further note that:
\begin{itemize}
\item
By Definitions~\ref{defdroplets-a}--\ref{defdroplets-b} and (H3-b), the only transitions in 
$\cX\starred$ between $\cX_\Box$ and $\cC^{++}$ are those where a free particle enters 
$\partial^-\Lambda$.
\item 
The only transitions in $\cX\starred$ between $\cC^{++}$ and $\cX\starred\backslash
[\cX_\Box\cup\cC^{++}]$ are those where the free particle moves from distance $2$ to 
distance $1$ of the protocritical droplet. All other transitions either involve a detachment 
of a particle from the protocritical droplet (which raises the number of droplets) or 
an increase in the number of particles in $\Lambda$. By (H3-b), such transitions lead to 
energy $>\Gamma\starred$, which is not possible in $\cX\starred$.
\item 
There are no transitions between $\cX_\Box$ and $\cX\starred\backslash[\cX_\Box\cup\cC^{++}]$.
\end{itemize}
The latter show that (\ref{red-3d.2d=2}) includes all the transitions in (\ref{redvp}).
\epr

%%%%%%%%%%%%%%%%

\subsubsection{Step 4: Sharp asymptotics for capacities of simple random walk}
\label{Step4}

With Lemma~\ref{red2} we have obtained upper and lower bounds on $\Theta$
in terms of capacities for simple random walk on $\Z^2$ of the pairs of sets 
$\partial^+\Lambda$ and $\CS(\hat{\eta})$, respectively, $\CS^{++}(\hat{\eta})$, 
with $\hat{\eta}$ summed over $\proto$. The transition rates of the simple random 
walk are 1 between neighboring pairs of sites. Lemma~\ref{capasymp} below, which 
is the analogue of \cite{BdHN06}, Lemma 3.4.1, shows that, in the limit as $\Lambda
\to\Z^2$, each of these capacities has the same asymptotic behavior, namely, 
$[1+o(1)]\,4\pi/\log|\Lambda|$, \emph{irrespective} of the location and shape 
of the protocritical droplet (provided it is not too close to $\partial^+\Lambda$, 
which is a negligible fraction of the possible locations). In what follows we 
pretend that $\Lambda=B_M=[-M,+M]^2\cap\Z^2$ for some $M\in\N$ large enough. It 
is straightforward to extend the proof to other shapes of $\Lambda$ (see van den 
Berg~\cite{vdB05} for relevant estimates).  

\bl{capasymp}
For any $\varepsilon>0$,
\be{capasymp*}
\begin{aligned}
&\lim_{M\to\infty}\,
\max_{{\hat{\eta}\in\proto} \atop {d(\partial^+B_M,\supp(\hat{\eta}))\geq\varepsilon M}}
\left|\frac{\log M}{2\pi}\,\CAPA^{\,B_M^+}\big(\partial^+B_M,\CS(\hat{\eta})\big)
-1\right| = 0,\\
&\lim_{M\to\infty}\,
\max_{{\hat{\eta}\in\proto} \atop {d(\partial^+B_M,\supp(\hat{\eta}))\geq\varepsilon M}}
\left|\frac{\log M}{2\pi}\,\CAPA^{\,B_M^+}\big(\partial^+B_M,\CS^{++}(\hat{\eta})\big)
-1\right| = 0,
\end{aligned}
\ee
where $d(\partial^+B_M,\supp(\hat{\eta}))=\min\{|x-y|\colon x\in\partial^+B_M,\,
y\in\supp(\hat{\eta})\}$.
\el
 
\bpr 
We only prove the first line of (\ref{capasymp*}). The proof of the second line is 
similar. 

\medskip\noindent
\underline{Lower bound}:
For $\hat{\eta}\in\proto$, let $y\in \CS(\hat{\eta}) \subset B_M$ denote the site closest 
to the center of $\CS(\hat{\eta})$. The capacity decreases when we enlarge the set over 
which the Dirichlet form is minimized. Therefore we have
\be{LB1}
\begin{aligned}
&\CAPA^{\,B_M^+}(\partial^+B_M,\CS(\hat{\eta}))
\geq \CAPA^{\,B_M^+}(\partial^+B_M,y)\\
&\qquad = \CAPA^{\,(B_M-y)^+}(\partial^+(B_M-y),0) 
\geq \CAPA^{\,B_{2M}^+}(\partial^+B_{2M},0),
\end{aligned}
\ee
where the last equality uses that $(B_M-y)^+ \subset B_{2M}^+$ because $y\in B_M$. By the 
analogue of (\ref{caprep}--\ref{capsym}) for simple random walk, we have (compare 
(\ref{red-3d.2d=2}) with (\ref{Diri}--\ref{capa}))
\be{capextrid}
\CAPA^{B_{2M}^+}(\partial^+B_{2M},0) = \CAPA^{B_{2M}^+}(0,\partial^+B_{2M})
= 4\, \P_0(\t_{\partial^+B_{2M}}<\t_0),
\ee
where $\P_0$ is the law on path space of the \emph{discrete-time} simple random 
walk on $\Z^2$ starting at 0. According to R\'ev\'esz \cite{R90}, Lemma 22.1, we 
have 
\be{LB3}
\P_0(\t_{\partial^+B_{2M}}<\t_0) \sim \frac{\pi}{2\log (2M)}, \qquad M\to\infty.
\ee 
Combining (\ref{LB1}--\ref{LB3}), we get the desired lower bound. 

\medskip\noindent
\underline{Upper bound}: 
As in (\ref{LB1}), we have
\be{UB4}
\begin{aligned}
&\CAPA^{\,B_M^+}(\partial^+B_M,\CS(\hat{\eta}))
\leq \CAPA^{\,B_M^+}(\partial^+B_M,S_y(\hat{\eta}))\\
&\qquad = \CAPA^{\,(B_M-y)^+}(\partial^+(B_M-y),S_y(\hat{\eta})-y)
\leq \CAPA^{\,B_{\varepsilon M}^+}(\partial^+B_{\varepsilon M},S_y(\hat{\eta})-y),
\end{aligned}
\ee
where $S_y(\hat{\eta})$ is the smallest square centered at $y$ containing $\CS(\hat{\eta})$,
and the last inequality uses that $(B_M-y)^+ \supset B_{\varepsilon M}^+$ when 
$d(\partial^+B_M,\supp(\hat{\eta}))\geq\varepsilon M$. By the recurrence of simple random walk, 
we have 
\begin{equation}
\label{UB5}
\CAPA^{\,B_{\varepsilon M}^+}(\partial^+B_{\varepsilon M},S_y(\hat{\eta})-y)
\sim \CAPA^{\,B_{\varepsilon M}^+}(\partial^+B_{\varepsilon M},0),
\qquad M \to\infty.
\end{equation}
Combining (\ref{LB3}--\ref{UB5}), we get the desired upper bound.
\end{proof}

Combining Lemmas~\ref{red2}--\ref{capasymp}, we find that $\Theta\in [\Theta_1,\Theta_2]$ 
with
\begin{equation}
\label{thetaiden}
\begin{aligned}
\Theta_1 &= O(\varepsilon M) + \sum_{{\hat{\eta}\in\proto} \atop 
{d(\partial^+B_M,\supp(\hat{\eta}))\geq\varepsilon M}}
\CAPA^{\,B_M^+}(\partial^+B_M,\CS(\hat{\eta}))\\
&= O(\varepsilon M) + \frac{2\pi}{\log M}\,\big|\big\{\hat{\eta}\in\proto\colon\, 
d(\partial^+B_M,\supp(\hat{\eta}))\geq\varepsilon M\big\}\big|\,[1+o(1)]\\
&= O(\varepsilon M) + \frac{2\pi}{\log M}\,N\starred\,[2(1-\varepsilon)M]^2\,[1+o(1)],
\end{aligned}
\end{equation}
and the same expression for $\Theta_2$, where we use that (recall (H3-a))
\begin{equation}
\CAPA^{\,B_M^+}\big(\partial^+B_M,\CS(\hat{\eta})\big)
\leq \CAPA^{\,B_M^+}\big(B_M^+\backslash\CS(\hat{\eta}),\CS(\hat{\eta})\big) 
= \tfrac12 |\CS^+(\hat{\eta})| \leq \tfrac12 (L\starred+2)^2, 
\end{equation}
and we recall from Definition~\ref{defdroplets-a}(b) that $N\starred$ is the cardinality 
of $\proto$ modulo shifts of the protocritical droplets. Let $M\to\infty$ followed by 
$\varepsilon\da 0$, to conclude that $\Theta \sim 2\pi N\starred (2M)^2/\log M$. Since 
$|\Lambda|=(2M+1)^2$ and $K=1/\Theta$, this proves (\ref{Kasymp}) in Theorem~\ref{tnucltime}.

%%%%%%%%%%%%%%%%%%%%%%%%%%%%

\subsection{Proof of Theorem~\ref{texp}}
\label{S2.4}

\bpr
The proof is immediate from Lemma~\ref{lemcapameta} and Bovier, Eckhoff, Gayrard and 
Klein~\cite{BEGK02}, Theorem 1.3(iv) and relies on (H1)--(H2) only. 
The main idea is that, each time the dynamics 
reaches the critical droplet but ``fails to go over the hill and falls back into the 
valley around $\Box$'', it has a probability exponentially close to 1 to return to 
$\Box$ (because, by (H2), $\Box$ lies at the bottom of its valley (recall \eqref{h*def} 
and \eqref{h*triv})) and to ``start from scratch''. Thus, the dynamics manages to grow a 
critical droplet and go over the hill to $\boxplus$ only after a number of  unsuccessful attempts 
that tends to infinity as $\beta\to\infty$, each having a small probability that
tends to zero as $\beta\to\infty$. Consequently, the time to go over the hill is
exponentially distributed on the scale of its average.
\epr

%%%%%%%%%%%%%%%%%%%%%%%%%%%%

\subsection{Proof of Theorem~\ref{tgate}}
\label{S2.5}

\begin{proof}
(a) The proof relies on (H1)--(H2) only.
We will show that there exist $C<\infty$ and $\d>0$ such that
\begin{equation}
\label{SVLTGeneralgate-1.0}
\P_\Box\left(\t_{\gate}<\t_\boxplus \mid \t_\boxplus<\t_\Box\right)
\geq 1-Ce^{-\d\b}, \qquad \forall\,\beta \in (0,\infty),
\end{equation}
which implies the claim.

By (\ref{caprep}), $\CAPA_\b(\Box,\boxplus) = \mu_\b(\Box)\,c_\b(\Box,
\cX\backslash\Box) \P_\Box(\tau_\boxplus<\tau_\Box)$ with $\mu_\b(\Box)=1/Z_\b$. 
From the lower bound in Lemma~\ref{lemmaprbds} it therefore follows that
\begin{equation}
\label{SVLTGenerallbcross}
\P_\Box(\tau_\boxplus<\tau_\Box) \geq C_1e^{-\Gamma\starred\b}\,
\frac{1}{c_\b(\Box,\cX\backslash\Box)}.
\end{equation} 
We will show that
\begin{equation}
\label{SVLTGeneralgate-1.1}
\P_\Box\left(\{\t_{\gate}<\t_\boxplus\}^c,\,\t_\boxplus<\t_\Box\right)
\leq C_2e^{-(\Gamma\starred+\d)\b}\,\frac{1}{c_\b(\Box,\cX\backslash\Box)}.
\end{equation}
Combining (\ref{SVLTGenerallbcross}--\ref{SVLTGeneralgate-1.1}), we get 
(\ref{SVLTGeneralgate-1.0}) with $C=C_2/C_1$. 

By Definitions~\ref{def2}(f) and \ref{def3}(d), any path from $\Box$ to $\boxplus$ 
that does not pass through $\gate$ must hit a configuration $\h$ with $ H(\h)>\Gamma\starred$. 
Therefore there exists a set $\cS$, with $H(\h)\geq \Gamma\starred+\d$ for all $\h\in\cS$ 
and some $\d>0$, such that
\begin{equation}
\label{SVLTGeneralgate-1.2}
\P_\Box\left(\{\t_{\gate}<\t_\boxplus\}^c,\,\t_\boxplus<\t_\Box\right)
\leq \P_\Box\left(\t_\cS<\t_\Box\right).
\end{equation}
Now estimate, with the help of reversibility,
\begin{equation}
\label{SVLTGeneralgate-1.3}
\begin{aligned}
\P_\Box\left(\t_\cS<\t_\Box\right)
&\leq \sum_{\h\in \cS} \P_\Box\left(\t_\h<\t_\Box\right)\\
&=\sum_{\h\in \cS} \frac{\mu_\b(\h)c_\b(\h,\cX\backslash\h)}
{\mu_\b(\Box)c_\b(\Box,\cX\backslash\Box)}\,
\P_\h\left(\t_\Box<\t_\h\right)\\ 
&\leq \frac{1}{c_\b(\Box,\cX\backslash\Box)}\sum_{\h\in \cS} 
|\{\h'\in\cX\backslash\h\colon\,\h\sim\h'\}|\,e^{-\b H(\h)}\\
&\leq \frac{1}{c_\b(\Box,\cX\backslash\Box)}\,C_2\,e^{-(\Gamma\starred+\d)\b}
\end{aligned}
\end{equation}
with $C_2=|\{(\h,\h')\in \cS\times\cX\backslash\h\colon\,\h\sim\h'\}|$,
where we use that $c_\b(\h,\h')\leq 1$. Combine 
(\ref{SVLTGeneralgate-1.2}--\ref{SVLTGeneralgate-1.3}) to get the claim in 
(\ref{SVLTGeneralgate-1.1}).

\medskip\noindent
(b) The proof relies on (H1) and (H3).
Write 
\begin{equation}
\label{SVLTGenerallbb1}
\P_\Box\big(\eta_{\t_{\entgate}} = \eta \mid \t_{\entgate}<\t_\Box\big)
=\frac{ \P_\Box\big(\eta_{\t_{\entgate}}=\eta,\,\t_{\entgate}<\t_\Box\big) }
{ \P_\Box\big(\t_{\entgate}<\t_\Box\big) },
\qquad \h\in\entgate.
\end{equation}
By reversibility,
\begin{equation}
\label{SVLTGeneralgate-2.1}
\begin{aligned}
\P_\Box\big(\eta_{\t_{\entgate}}=\eta,\,\t_{\entgate}<\t_\Box\big)
&=\frac{ \mu_\b(\eta)c_\b(\h,\cX\backslash\h)}
{\mu_\b(\Box)c_\b(\Box,\cX\backslash\Box) }\,
\P_\eta\big(\t_\Box<\t_{\entgate}\big)\\
&= e^{-\Gamma\starred\b}\,\frac{c_\b(\h,\cX\backslash\h)}
{c_\b(\Box,\cX\backslash\Box)}\,
\,\P_\eta\big(\t_\Box<\t_{\entgate}\big),
\qquad \h\in\entgate.
\end{aligned}
\end{equation}
Moreover (recall (\ref{h*def}--\ref{soleq})),
\begin{equation}
\label{SVLTGeneralequimeas}
\P_\eta\big(\t_\Box<\t_{\entgate}\big)
=\sum_{{\eta'\in\cX\backslash\entgate} \atop {\h\sim\h'}} 
\frac{c_\b(\eta,\eta')}{c_\b(\h,\cX\backslash\h)}\,
h\starred_{\Box,\entgate}(\eta'),
\qquad \h\in\entgate,
\end{equation}
where 
\begin{equation}
\label{SVLTGeneralh*rels}
h\starred_{\Box,\entgate}(\h')
= \left\{\begin{array}{ll}
0 &\mbox{if } \h'\in\entgate,\\
1 &\mbox{if } \h'=\Box,\\
\P_{\h'}(\t_\Box<\t_{\entgate}) &\mbox{otherwise}.
\end{array}
\right.
\end{equation} 
Because $\proto\subset\cX_\Box$ by Lemma~\ref{lemmavalleys} and $\entgate\subset
\cG(\Box,\boxplus)$ by Definition~\ref{defdroplets-a}(a), for all $\h'\in\proto$
we have $\comlev(\h',\entgate)- \comlev(\h',\Box) = \Gamma\starred - \comlev(\h',\Box)
\geq \delta>0$. Therefore, as in the proof of Lemma~\ref{lh*est}, it follows that
\begin{equation}
\label{SVLTGeneralbdvar1}
\min_{\h'\in\proto} h\starred_{\Box,\entgate}(\h')
\geq 1-Ce^{-\d\b},
\end{equation}
Moreover, letting $\bar{\cC}\starred$ be the set of configurations that can be reached 
from $\entgate$ via an allowed move that does not return to $\proto$, we have  
\begin{equation}
\label{SVLTGeneralbdvar2}
\max_{\h'\in\bar{\cC}\starred} h\starred_{\Box,\entgate}(\h') \leq Ce^{-\d\b}.
\end{equation}
Indeed, $h\starred_{\Box,\entgate}(\h')=0$ for $\h'\in\entgate$, while we have
the following:
\begin{equation}
\label{trpr}
\text{Any path from } \bar{\cC}\starred\backslash\entgate \text{ to } \Box
\text{ that avoids } \entgate \text{ must reach an energy level } >\Gamma\starred.
\end{equation}
To obtain \eqref{SVLTGeneralbdvar2} from \eqref{trpr}, we can do an estimate similar 
to (\ref{est1}--\ref{est1a}) for $\h'\in\bar{\cC}\starred\backslash\entgate$.

To prove \eqref{trpr} we argue as follows. Let $\zeta\in\bar{\cC}\starred$, and 
let $\eta$ be the configuration in $\entgate$ from which $\zeta$ is obtained in a single 
transition. If $\zeta \in \entgate$, then any path from $\zeta$ to $\Box$ already starts 
from $\entgate$ and there is nothing to prove. Therefore, let $\zeta\in\bar{\cC}\starred
\backslash \entgate$. Note that, by (H3-a), $\eta$ consists of a single (protocritical) 
droplet in $\Lambda^{-}$ plus a particle of type $\tb$ in $\partial^-\Lambda$. Recalling 
that particles in $\partial^-\Lambda$ do not interact with other particles, we see that 
any configuration obtained from $\eta$ by detaching a particle from the (protocritical) 
droplet increases the number of droplets and, by (H3-b), raises the energy above $\Gamma
\starred$. Therefore, $\zeta$ can only be obtained from $\eta$ by moving the free particle 
from $\partial^-\Lambda$ to $\Lambda^-$. Only two cases are possible: either $\zeta\in
\gate_{\mathrm{att}}$ or $\zeta \in \gate \backslash \entgate$. In the former case, the 
claim follows via Lemma~\ref{lbacktrack}. In the latter case, we must show that if there 
is a path $\omega\colon\,\zeta\to\Box$ that avoids $\entgate$ such that $\max_{\sigma\in\omega} 
H(\sigma) \leq \Gamma\starred$, then a contradiction occurs. 

Indeed, if $\omega$ is such a path, then the reversed path $\omega'$ is a path from 
$\Box\to\zeta$ such that $\max_{\sigma\in\omega'} H(\sigma) \leq \Gamma\starred$. But
$\omega'$ can be extended by the single move from $\zeta$ to $\eta$ to obtain a path 
$\omega''\colon\,\Box \to \eta$ such that $\max_{\sigma\in\omega''} H(\sigma) \leq 
\Gamma\starred$. Moreover, since $\eta\in\entgate$, there exists a path $\gamma\colon\,
\eta\to\boxplus$ such that $\max_{\sigma\in\gamma} H(\sigma) \leq \Gamma\starred$. But 
then the path obtained by joining $\omega''$ and $\gamma$ is a path in $(\Box\to
\boxplus)_{\mathrm{opt}}$ such that the configuration $\zeta$  visited just before 
$\eta\in\entgate$ belongs to $\gate \backslash \entgate \subset \gate$. However, by 
Definitions~\ref{defdroplets-a}--\ref{defdroplets-b}, this implies that $\zeta\in\proto$,
which is impossible because $\proto \cap \gate = \emptyset$.  

The estimates in (\ref{SVLTGeneralbdvar1}--\ref{SVLTGeneralbdvar2}) can be used as follows.
By restricting the sum in (\ref{SVLTGeneralequimeas}) to $\h'\in\proto$ and inserting 
(\ref{SVLTGeneralbdvar1}), we get
\begin{equation}
\label{SVLTGeneralgate-3.1}
\P_\h\big(\t_\Box<\t_{\entgate}\big)
\geq (1-Ce^{-\d\b})\,\frac{c_\b(\h,\proto)}{c_\b(\h,\cX\backslash\h)},
\qquad \h\in\entgate.
\end{equation}
On the other hand, by inserting (\ref{SVLTGeneralbdvar2}), we get
\begin{equation}
\label{SVLTGeneraltriv-upper}
\P_\h\big(\t_\Box<\t_{\entgate}\big)
\leq \frac{c_\b(\h,\proto)}{c_\b(\h,\cX\backslash\h)} 
+ Ce^{-\d\b}|\bar{\cC}\starred|,
\qquad \h\in\entgate.
\end{equation}
Because $H(\proto)<H(\entgate)=\Gamma\starred$, we have
\begin{equation}
\label{countprotcritcross}
c_\b(\h,\proto) = \sum_{\h'\in\proto} c_\b(\h,\h') 
= |\{\h'\in\proto\colon\,\h\sim\h'\}|, \qquad \h\in\entgate, 
\end{equation}
and, since $c_\b(\h,\cX\backslash\h)\leq |\cX|$, it follows that $\h \mapsto 
c_\b(\h,\proto)/c_\b(\h,\cX\backslash\h)$ is bounded from below. Combine this 
observation with (\ref{SVLTGeneralgate-3.1}--\ref{SVLTGeneraltriv-upper}), to get 
\begin{equation}
\label{SVLTGeneraltriv-asymp}
\P_\h\big(\t_\Box<\t_{\entgate}\big)
= [1+O(e^{-\d\b})]\,\frac{c_\b(\h,\proto)}{c_\b(\h,\cX\backslash\h)},
\qquad \h\in\entgate.
\end{equation}
Combining this in turn with (\ref{SVLTGenerallbb1}--\ref{SVLTGeneralgate-2.1}), we 
arrive at
\begin{equation}
\label{SVLTGeneraltrivfinal}
\begin{aligned}
\P_\Box\big(\eta_{\t_{\entgate}}=\eta \mid \t_{\entgate}<\t_\Box\big)
&= \frac{c_\b(\h,\cX\backslash\h)\,\P_\h(\t_\Box<\t_{\entgate})}
{\sum_{\h'\in\entgate} c_\b(\h',\cX\backslash\h')\,\P_{\h'}(\t_\Box<\t_{\entgate})}\\
&= [1+O(e^{-\d\b})]\,\frac{c_\b(\h,\proto)}{\sum_{\h'\in\entgate}
c_\b(\h',\proto)}, \qquad \h\in\entgate.
\end{aligned}
\end{equation}
Finally, each site in $\partial^-\Lambda$ has one edge towards $\partial^+\Lambda$ 
and hence, by (\ref{countprotcritcross}), $\eta \mapsto c_\b(\h,\proto)$ is constant 
on $\entgate$. Together with (\ref{SVLTGeneraltrivfinal}) this proves the claim.
\end{proof}

%%%%%%%%%%%%% REFERENCES %%%%%%%%%%%%%%%%%%%%%%%%%%%%%%%%%%%%%%%%%%%%%%%%%%%%%%%%%

\end{document}